\documentclass{siamart}
\usepackage{amsmath,amssymb}
\usepackage{graphicx}
\usepackage[boxed,algo2e]{algorithm2e} 

\newcommand{\R}{\mathbb{R}}

\newcommand{\define}{\equiv}
\newcommand{\normtwo}[1]{\| #1\|_2}
\newcommand{\normf}[1]{\| #1\|_\text{F}}

\newcommand{\mat}[1]{{#1}}
\renewcommand{\vec}[1]{{#1}}
\newcommand{\vech}[1]{\widehat{{#1}}}

\newcommand{\diag}{\mathsf{diag}}

\newcommand{\mc}[1]{\mathcal{#1}}
\newcommand{\mb}[1]{\mathbb{#1}}
\newcommand{\mch}[1]{\widehat{\mathcal{#1}}}

\newcommand{\Wh}{\widehat{W}}
\newcommand{\Ah}{\widehat{A}}
\newcommand{\rank}{\mathsf{rank}\,}
\newcommand{\deim}{\widehat{f}_\text{DEIM}}
\newcommand{\rdeim}{\widehat{f}_\text{R-DEIM}}

\newcommand{\D}{\mathbb{D}}
\newcommand{\Dh}{\widehat{\mathbb{D}}}

\newcommand{\bmat}[1]{\begin{bmatrix}#1\end{bmatrix}}

\newcommand{\Sh}{\widehat{S}}

\newcommand{\range}{\mc{R}}

\newcommand{\cls}{C_\text{LS}}

\newcommand{\changes}[1]{{#1}}

\Crefname{prop}{Proposition}{Propositions}
\Crefname{algocf}{Algorithm}{Algorithms}

\newtheorem{remark}{Remark}
\title{Randomized Discrete Empirical Interpolation Method for Nonlinear Model Reduction}
\author{Arvind K. Saibaba \thanks{Department of Mathematics, North Carolina State University asaibab@ncsu.edu. This work was funded, in part, by the National Science Foundation through the awards DMS 1720398 and 1821149.}} 

\begin{document} 
\maketitle

\begin{abstract}  Discrete empirical interpolation method (DEIM) is a popular
technique for nonlinear model reduction and it has two main ingredients: an
interpolating basis that is computed from a collection of snapshots of the solution,
and a set of indices which determine the nonlinear components to be simulated.
The computation of these two ingredients dominates the overall cost of
the DEIM algorithm. To specifically address these two issues, we present
randomized versions of the DEIM algorithm. There are three main contributions
of this paper. First, we use randomized range finding algorithms
to efficiently find an approximate DEIM basis. Second, we develop
randomized subset selection tools, based on leverage scores, to efficiently
select the nonlinear components. Third, we develop  several 
theoretical results that quantify the accuracy of the randomization on the DEIM
approximation. We also present numerical experiments that demonstrate the
benefits of the proposed algorithms.  
\end{abstract}

\section{Introduction} Detailed mathematical models of weather prediction and
neuroscience routinely generate large-scale problems having over a billion
unknowns.  The goal of model reduction is to replace
computationally expensive full-scale models by reduced order models (ROMs) that
are cheaper to evaluate and that preserve the important underlying physics in
the full-scale models.  Development of effective ROMs will enable efficient and
accurate simulation of a wide range of detailed complex physical phenomena, as
well as benefit a host of applications in inverse problems, data assimilation,
design, control, optimization, and uncertainty quantification.

In a typical model reduction technique, there are two distinct phases: an
\textit{offline phase} in which the full model is simulated for a range of
parameters or specifications and the outputs of this simulations are used to construct the ROM,
and an \textit{online phase}, in which the ROM is simulated for the desired
parameter, or specification. A successful ROM has two features that are hard to
achieve simultaneously: the ROM should be accurate over the desired range of
parameters and specifications, and the online phase should be inexpensive and
the dominant computational cost should be in the offline phase.  A popular
method for model reduction is the proper orthogonal decomposition (POD), which
has been successfully used in a wide range of partial differential equation
(PDE)-based applications, and is reviewed in \cref{ssec:pod}. While POD
approach has broad applicability, the computational efficiency of the POD
demands that the underlying PDE has to be linear, or the parameter dependence
has to be of a specific type, e.g., affine parameter dependence. To address
this deficiency, many methods have been proposed in the literature such as
gappy POD interpolation method, empirical interpolation method (EIM), and its
discrete variant, the discrete empirical interpolation method (DEIM). Reviews 
of various model reduction techniques are provided in the survey
paper~\cite{benner2015survey}, recent books and
monographs~\cite{hesthaven2016certified,quarteroni2015reduced}.

The DEIM interpolation framework computes an approximation of a nonlinear
function $f: \mb{R}^n \rightarrow \mb{R}^n$ by the means of a basis $W \in
\R^{n\times r}$ used to interpolate the function, and a set of indices,
defining a point selection operator $S \in \R^{n\times r}$, at which the
nonlinear function is evaluated. In \cref{ssec:deim}, we explain how $W$ and
$S$ can be used to approximate the function $f$; here, we describe the major
bottlenecks in computing the DEIM approximation.

\begin{itemize} 
\item The basis $W$ is constructed as follows: several
representative samples---also called as snapshots---of the function $f(\cdot)$
are collected and arranged as columns of a matrix, known as the snapshot matrix. The left singular vectors of
this snapshot matrix form the desired DEIM basis---henceforth, we call this the
standard basis. The dimension of the subspace spanned by the DEIM basis, denoted by $r$, depends on the number of
dominant singular values of the snapshot matrix. Computing a truncated singular value decomposition (SVD)
costs $\mc{O}(rn_sn)$, where  $n_s$ is the number of snapshots;  our approach replaces the SVD by a
randomized SVD. 

\item Finding a  set of good indices is a  combinatorially hard problem known
as {\em subset selection} (in the DEIM literature, this is known as point
selection, which we also adopt in this manuscript). Various deterministic
subset selection techniques have been proposed in the literature: based on
pivoted LU factorization~\cite{chaturantabut2010nonlinear,sorensen2016deim},
pivoted QR factorization~\cite{drmac2016new}, strong rank-revealing QR
factorization~\cite{drmac2017discrete}. The computational cost is roughly
$\mc{O}(nr^2)$; we use randomized subset selection techniques which lowers this
cost and can exploit parallelism.  \end{itemize}	

 Both of these operations are computationally expensive when $n$ is large. Our
paper specifically addresses these computational challenges using randomized
algorithms, thereby enabling efficient large-scale implementation of DEIM.

To motivate the development of randomized algorithm for DEIM, we briefly review
randomized algorithms in other applications. Recently, randomized algorithms
have been developed for accurate low-rank approximations to matrices arising
from large datasets. The basic idea of these methods is to use random sampling
to identify a subspace which approximately captures the range of the
matrix~\cite{halko2011finding}. The matrix is then projected onto this subspace
and then deterministic linear algebraic techniques can be used to manipulate
the projected matrix to obtain the desired low-rank approximation.   Besides
low-rank approximations, randomized methods (based on leverage score sampling
and other subset selection techniques) have been developed for other linear
algebraic problems such as  least squares problems, regression, and computation
of interpretable decompositions such as CX/CUR decompositions, see recent
survey articles~\cite{mahoney2011randomized,drineas2016randnla}.  Randomized
methods have several advantages over their corresponding classical
counterparts: typically, they are computationally efficient, numerically
robust, easy to implement, scale well in a distributed computing setting, and
have well-developed error analysis.

\paragraph{Contributions and Contents} We present randomized algorithms for
DEIM that enable nonlinear dimensionality reduction for several large-scale applications. 

We present randomized approaches (\cref{sec:randbasis}) for efficient
construction of a DEIM basis $\Wh$. The algorithms come in two flavors depending on
whether the target rank $r$ is known or not. When the target rank is known, we
present a basic version and a more accurate version based on subspace
iteration. When the target rank $r$ is unknown, we provide an adaptive
algorithm for computing $\Wh$. The computational and storage
advantages of the randomized algorithms for constructing the DEIM basis are
highlighted. 

We present a detailed analysis of the error (\cref{sec:error}) when a
randomized basis instead of the standard DEIM basis. A crucial component of our
analysis involves the largest canonical angle between the subspaces spanned by
the standard basis $W$, and an approximate basis $\Wh$. This analysis is
applicable to any approximate basis $\Wh$ and therefore this has broad appeal
beyond the context of randomized algorithms. We also present specific results
that explicitly account for the randomness on the accuracy of the DEIM
approximation.

We propose two randomized point selection methods for DEIM approximation
(\cref{sec:point}): leverage score sampling and hybrid algorithms. For each
method, we present theoretical bounds on the number of points required for the desired
accuracy. The hybrid point selection technique combines the computational
advantages of the randomized methods with the accuracy of deterministic
methods. These sampling methods have been proposed in the context of matrix CUR
decompositions; our analysis for the DEIM approximation is new. 

Numerical experiments (\cref{sec:num}) demonstrate the computational benefits,
the accuracy of the randomized approaches and insight into the choice of
parameters for various algorithms presented in this paper. 

\paragraph{Related work} The idea of using randomization to accelerate
computations in model reduction appears to be relatively new and here we briefly
review the literature. Randomized matrix methods, similar to the ones used in
this paper, have been used to approximate 
POD~\cite{yu2015randomized,yu2016computationally,bach2019fixed}, and dynamic mode
decomposition~\cite{erichson2016randomized,erichson2017randomized,bistrian2017randomized}.
Recent work in~\cite{aslan2017randomized} uses randomization for reducing the
cost associated with multiple right hand sides in nonlinear model reduction.
The resulting ROM dramatically reduces the cost of solving a PDE-based inverse
problem. However, none of these references directly tackle the DEIM
approximation, which is the central focus of this paper. Randomized sampling
approaches for choosing the DEIM indices have been proposed
in~\cite{drmac2016new} but no analysis of the randomization was presented.
Another noteworthy paper~\cite{peherstorfer2018stabilizing}, uses randomized
oversampling to address stability issues with DEIM.

\section{Preliminaries}
We briefly review the POD and DEIM approaches for model order reduction. 
\subsection{Proper Orthogonal Decomposition}\label{ssec:pod}
We explain the POD approach in the context of a nonlinear dynamical system that takes the form
\begin{equation}\label{e_dynamical}
\frac{d\vec{x}(t)}{dt} = \mat{M}\mat{x}(t) + \vec{f}(\vec{x}(t)), \qquad x(0) = x_0
\end{equation}
where $\mat{M}  \in \R^{n\times n}$. When $n$ is large, the simulation of the dynamical system can be computationally expensive, and we turn to reduced order models to lower the computational cost. In the POD approach, the dynamical system is first simulated numerically and the ``snapshots'' of the system at multiple times $ 0  \leq t_1 < \dots < t_{n_t} \leq T$, as $\vec{x}_j = \vec{x}(t_j) $ for $j=1,\dots,n_t$ are collected into the POD snapshot matrix 
\[ \mat{F} = \bmat{\vec{x}_1 & \dots & \vec{x}_{n_t}} \in \R^{n\times n_t}. \]
We then  compute its thin SVD $ \mat{F} = \mat{Y}\mat{\Sigma }\mat{Z}^\top$. The POD approach chooses the $k \leq \rho = \rank(F)$ singular vectors corresponding to the largest singular values, collected in a matrix $\mat{V}_k = \mat{Y}(:,1:k)$ (using MATLAB notation). The basis, thus obtained, also solves the following optimization problem
\begin{equation}
\min_{{V} \in \R^{n\times k}: {V}^\top{V}= \mat{I}_k} \normf{ ({I}_n - VV^\top)\mat{F} }^2   =  \sum_{j=k+1}^{n_t} \sigma_j^2(F),
\end{equation}
where $\sigma_j(F)$ are the singular values of $F$. 
The choice of snapshots is an important issue in determining an effective POD basis and is discussed in~\cite{quarteroni2015reduced,benner2015survey,hesthaven2016certified}. Assuming that we have an effective basis $\mat{V}_k$, the solution $\vec{x}(t)$ can be approximated to be constrained in the span of the basis of the columns of $\mat{V}_k$, i.e., $\vec{x}(t) \approx \mat{V}_k\vech{x}(t)$. Next, the reduced system is obtained by a Galerkin projection onto $\range(\mat{V}_k)$; the dynamics of $\vech{x}(t)$ is given by 
\[\frac{d\vech{x}(t)}{dt} = \mat{V}_k^\top \mat{M}\mat{V}_k\vech{x}(t) + \mat{V}_k^\top\vec{f}(\mat{V}_k\vech{x}(t)), \qquad \vech{x}(t) = V_k^\top x_0.\]
This approach, known as the {\em POD-Galerkin} method, is an effective way of reducing the dimensionality of the system of equations represented in~\cref{e_dynamical}. This approach is more generally applicable to other applications such as parameterized PDEs.

We briefly discuss the computational cost of the POD-Galerkin approach. Note
that the matrix $\widehat{\mat{M}} \define \mat{V}_k^\top \mat{M}\mat{V}_k$ can
be precomputed. If the nonlinear term $f = 0$, then the cost of simulating the
reduced system for $\vech{x}(t)$ is independent of $n$, the dimension of the
full order system. However, the evaluation of the nonlinear term $V_k^\top
f(\mat{V}_k\vech{x}(t))$ has computational complexity that depends on $n$. As a
result, evaluation of this term may still be as expensive as solving the
original system. 

\subsection{DEIM approximation}\label{ssec:deim} The DEIM approximation was
proposed to address the deficiency of POD-Galerkin approach for nonlinear
dynamical systems. Given a collection of snapshots of the full order dynamical
system 
$$A = \bmat{f(\vec{x}_1) & \dots & f(\vec{x}_{n_s})},$$ 
a projection basis is computed by retaining the $r$ left-singular vectors
corresponding to the top-$r$ singular values. We denote this standard DEIM
basis by $\mat{W} \in \R^{n\times r}$. The DEIM approach then selects $s$
distinct rows from the matrix $\mat{W}$; in particular, denoting the row
indices $\{t_1,\dots,t_s\}$ we define the {\em selection operator}
\[ S =  \bmat{\vec{e}_{t_1} &  \cdots & \vec{e}_{t_s}} \in \R^{n\times s} ,\]
where $\vec{e}_{t_i}$ is the $t_i$-th column of the $n\times n$ identity matrix.

We define the DEIM projector $\D \define W (S^\top W)^{\dagger}S^\top$ that satisfies $\D^2 = \D$ (i.e., $\D$ is idempotent) and
if $\D \neq 0$ and $\D \neq I_n$, then 
\begin{equation}\label{oblique}
	\normtwo{\D} = \normtwo{ \mat{I}_n - \D} .
\end{equation}
The equality follows from the property of oblique projectors~\cite{szyld2006many}. Since $\mat{S}$ and $\mat{W}$ have orthonormal columns and $\normtwo{\cdot}$ is unitarily invariant, we also have the equality $\|\D\|_2 =  \normtwo{(\mat{S}^\top\mat{W})^{\dagger}}$. 
Given the DEIM projector, the DEIM approximation of $\vec{f}$ can be expressed as 
\begin{equation} \label{eqn:deimapprox}
\deim \define \D f . 
\end{equation}

We recapitulate a few properties of the DEIM projector that will be useful in subsequent analysis. Let $P_W \equiv WW^\top$ and $P_S = SS^\top$ be two orthogonal projectors corresponding to the bases $W$ and $S$ respectively. Suppose that $\rank(\D) = \rank(S^\top W) = r = s$. The following relations hold 
\begin{align} \label{eqn:obl0}
\D = & \>  \D P_S = P_W\D, \\ \label{eqn:obl1}
\D P_W = & \> P_W,  &  I_n -\D = & (I_n-\D)(I_n-P_W), \\ \label{eqn:obl2}
P_S\D = & \> P_S, & I_n - \D =& (I_n-P_S)(I_n-\D).
\end{align}

\begin{remark} {In the original DEIM approach~\cite{chaturantabut2010nonlinear},
the number of selected indices $s$ equals the dimension of the DEIM basis $r$.
The case $s \neq r$ has its origins {in gappy POD (see~\cite[Remark 10.4]{quarteroni2015reduced} and references therein). The implications for the interpolation and projection properties of the DEIM operator have been discussed  in detail in~\cite[section 3]{drmac2017discrete}}.
Of importance in our analysis will be the case $s \geq r$ and $\rank(\D) = r$.
In this case, \cref{eqn:obl0,eqn:obl1} hold, but \cref{eqn:obl2} no longer
holds, see~\cite[section 3]{drmac2017discrete}. }  \end{remark}

We also have the following the error in the DEIM approximation. 
\begin{lemma} \label{l_deim} Let the DEIM approximation be defined in~\cref{eqn:deimapprox}, and let $\rank(\D) = \rank(S^\top W) = r$ with $s\geq r$ 
	\[ \normtwo{\vec{f} - \deim} \leq \|\D\|_2 \normtwo{(I_n-WW^\top) {f}} .\]
\end{lemma}
\begin{proof}
{See \cite[Lemma 3.2]{chaturantabut2010nonlinear}. }
\end{proof}

\begin{remark}\label{rem:sigmar}
The error in the function approximation $\normtwo{(I_n-WW^\top) {f}}$ depends on the quality of the basis $W$. While \textit{a priori} bounds for this term are difficult to derive. {In this work,  we make the rather strong assumption that $\normtwo{(I_n-WW^\top) {f}} \approx \sigma_{r+1}(A)$ (see~\cite[Equation (10.20)]{quarteroni2015reduced} and the discussion surrounding it for a justification)}.  That is, the error depends on the largest discarded singular value of the snapshot matrix. 
\end{remark}

The quantity $\|\D\|_2 $ can be interpreted as the condition number of the DEIM
approximation. In what follows, we also refer to it as the DEIM error constant
and it depends on the particular point selection technique that is used. A
discussion of this is provided in~\cref{sec:point}. When $S$ and $W$ have
orthonormal columns and $\rank(S^\top W) = r$ then $\|\D\|_2 = \normtwo{(S^\top W)^\dagger}$.

It remains to be shown, how to use the DEIM approximation along with the POD.
Suppose that the DEIM basis $W$ and the selection operator $S$ have been
determined. In the dimension reduced version of the nonlinear dynamical system,
replace $f$ by $\deim$, i.e.,
\[ \mat{V}_k^\top\vec{f}(\mat{V}_k\vech{x}(t)) \approx  (\mat{V}_k^\top W) (S^\top W)^{\dagger} S^\top\vec{f}(\mat{V}_k\vech{x}(t)).\]
In the offline stage, the matrices $(\mat{V}_k^\top W)$ and a factorization of
$S^\top W$ can be precomputed. This involves a computational cost of
$\mc{O}(nkr + r^3)$. In the online stage, instead of evaluating the nonlinear
function, only $s \geq r$ selected number of components of the function will be
evaluated at the indices that determine the
columns of the  selection operator $S$.

The computation of {a standard DEIM basis} $W$ and the interpolating indices
that define $S$ constitute the major bottlenecks in the large-scale
implementation of DEIM. In \cref{sec:randbasis}, we develop randomized
algorithms to accelerate the computation of the basis $W$ and in
\cref{sec:point}, we develop efficient randomized point selection algorithms.

\section{Randomized algorithm for DEIM basis}\label{sec:randbasis}
 Let $A \in \R^{n\times n_s}$ with $n \geq n_s$  and $r \leq \rank(A)$ be the matrix of snapshots. We can write the SVD of $A$ as 
\[ A = \bmat{W_1 & W_2} \bmat{ \Sigma_1 \\ & \Sigma_2} \bmat{Z_1^\top \\ Z_2^\top}. \]
Here $W_1 \in \R^{n\times r}$ and $W_2 \in \R^{n\times (n -r)}$ contain the
left singular vectors, whereas $Z_1 \in \R^{n_s\times r}$ and $Z_2 \in
\R^{n_s\times (n_s -r)}$ contain the right singular vectors. The matrices
$\Sigma_1 = \diag(\sigma_1,\dots,\sigma_r)\in \R^{r\times r}$ and $\Sigma_2 =
\diag(\sigma_{r+1},\dots,\sigma_{n_s}) \in \R^{(n-r)\times (n_s-r)}$ contain
the singular values of $A$ in decreasing order.  Furthermore,
$\|\Sigma_1^{-1}\|_2 = 1/\sigma_r$ and $\|\Sigma_2\|_2 = \sigma_{r+1}$. Write
$A = A_r + A_{r,\perp}$, where by the Eckart-Young theorem, $A_r \equiv
W_1\Sigma_1 Z_1^\top$ is the best rank-$r$ approximation to $A$~\cite[section 4.5]{stewart1990matrix}.

The standard DEIM approximation uses $W=W_1$. The randomized DEIM approximation
replaces the exact SVD with a randomized SVD. Here and henceforth, we refer to
the basis generated using a randomized algorithm as the R-DEIM basis, and the
resulting approximation as the R-DEIM approximation.

\subsection{Randomized SVD} We briefly review an idealized version of the
randomized range finding algorithm. Suppose the target rank is known and
denoted by $r$. Draw a standard Gaussian random matrix $\Omega \in \R^{n\times
r}$ (i.e., a matrix with entries independent and identically distributed normal variables having 
mean $0$ and variance $1$).  Form the matrix $Y = A\Omega$ (which is often
called the ``sketch matrix'' or ``sketch''), compute an orthonormal basis for
$\mc{R}(Y)$ using the thin QR factorization $Y= QR$. In practice, instead of
$r$ columns, $\ell = r + p$ columns are drawn; here $p > 0$ is a small
oversampling parameter. A low-rank approximation to $A$ can be obtained as $A
\approx Q(Q^\top A)$.  Compute the top $r$ left singular vectors of $B \equiv
Q^\top A$ by $W_r$, and obtain the approximate basis as $\Wh = QW_r$. This is
summarized in \cref{alg:basic}.  

\LinesNumbered
\begin{algorithm2e}[!ht]
\DontPrintSemicolon
\SetKwInput{Input}{Input}
\SetKwInput{Output}{Output}
	\Input{ Snapshot matrix $A \in \R^{n\times n_s}$, target rank $r$, oversampling parameter $p \geq 1$ such that $r + p \leq \min\{n,n_s\}$. }
	\Output{ Basis $\Wh$ with orthonormal columns.}
	 Draw a standard Gaussian matrix $\Omega \in \R^{n_s\times (r+p)}$.\;
	 Form $Y = A\Omega$ and compute thin-QR factorization $Y = QR$. \;
	 Form $B = Q^\top A$. Let $W_r$ be the left singular vectors corresponding to top $r$ singular values of $B$. \;
	Form $\Wh \leftarrow QW_r$. \;
\caption{Basic randomized range-finding algorithm}
\label{alg:basic}
\end{algorithm2e}
The error in the low-rank approximation can be obtained by~\cite[Theorem 10.6]{halko2011finding} (when $p \geq 2$ and $n \geq n_s$):   
\[ \mb{E}_\Omega\,\| (I_n-QQ^\top)A\|_2 \leq \left(1 + \sqrt{\frac{r}{p-1}}\right) \sigma_{r+1} +  \frac{e\sqrt{r+p}}{p}\left(\sum_{j=r+1}^{n_s}\sigma_j^2 \right)^{1/2} .\]
However, the error in the low-rank representation does not fully explain the
error in the R-DEIM approximation. As shown in \cref{sec:error}, we will need
to bound the canonical angles between the subspaces spanned by the singular
vectors.  

In practice, the target rank $r$ may not be known in advance. In the standard
DEIM approach, the rank $r$ is chosen based on the decay of the singular values
of $A$. However, the exact singular values are not known to us and therefore, a
new approach is needed. We use the approach in~\cite{gu2016efficient} that
adaptively determine the target rank $r$.

\subsection{Adaptive randomized range finder}\label{ssec:adapt}

In the standard DEIM approach, the truncation index is taken to be the smallest
index $r$ which satisfies \[ \frac{\|A-A_r\|_F^2}{\|A\|_F^2} =
\frac{\sum_{j=r+1}^{n_s}\sigma_{j}^2(A)}{\sum_{j=1}^{n_s}\sigma_{j}^2(A)} \leq
\epsilon_\text{tol}^2. \] 
{The matrix $A_r$ was defined at the start of this section. }
This condition ensures that the relative error of the
low-rank representation, as measured in the Frobenius norm, is smaller than a
positive user-defined tolerance $\epsilon_\text{tol}$. Based on this criterion,
we require the R-DEIM basis $\Wh$ to satisfy the condition
\begin{equation}\label{eqn:alterror} 
\|A - \Wh\Wh^\top A\|_F^2 \leq\epsilon_\text{tol}^2\|A\|_F^2.  
\end{equation}
Several adaptive randomized range finding algorithms were presented in the
literature~\cite{halko2011finding,martinsson2016randomized,gu2016efficient,gorman2018matrix}.
In this paper, we adopt the approach presented in~\cite{gu2016efficient}. An
outline of this algorithm is given in \cref{alg:rangetwo}; however, we refer
the reader to~\cite[Algorithm 3]{gu2016efficient} for details regarding the
implementation. Another variation \cite[Algorithm 4]{gu2016efficient} combines
an adaptive strategy with the subspace iteration for enhanced accuracy. 

\LinesNumbered
\begin{algorithm2e}[!ht]
\DontPrintSemicolon
\SetKwInput{Input}{Input}
\SetKwInput{Output}{Output}
\Input{ Snapshot matrix $A \in \R^{n\times n_s}$, block size $b \geq 1$, tolerance $\epsilon_\text{tol} > 0$, factor $\alpha$.}
\Output{Basis $\Wh$ with orthonormal columns.}
 Compute $\alpha = \|A\|_F^2$.\; 
 $W \leftarrow  [] $. \;
 Parameter $\beta=0$.	 \tcp{Norm of $B=Q^\top A$.} 
\While {$\beta \leq \alpha (1-\epsilon_\text{tol}^2)$} {
 Draw a standard Gaussian matrix $\Omega \in \R^{n_s\times b}$. \; 
 Compute $Z = (A - WB)\Omega $, and the thin-QR factorization $QR = Z$. \;
 Orthogonalize: $Q \leftarrow (I_n-WW^\top)Q$. \;
 Compute $B' = Q^\top A - Q^\top WB$. \;
 Extend $W = \bmat{W & Q}$ and $B = \bmat{ B \\ B'}$. \;
 Update norm: $\beta \leftarrow \beta + \|Q^\top A\|_F^2$. \;
}
$\Wh \leftarrow W$.\;
\caption{Adaptive randomized range-finding algorithm \cite{gu2016efficient}}
\label{alg:rangetwo}
\end{algorithm2e}

\subsection{Computational advantages} The standard DEIM approach computes the
compact SVD of the snapshot matrix; this is expensive and costs $\mc{O}(n
n_s^2)$ flops, assuming that $n_s \leq n$. On the other hand the randomized
approach only requires $\mc{O}(rn n_s)$ flops and is computationally
advantageous when $r \ll n_s$.

Besides this, there are several benefits of this approach that are worth
pointing out. First, the rank $r$ need not be known a priori and is determined
adaptively (\cref{alg:rangetwo}). Second,  the  sketch $Y = A\Omega$ can be
computed by taking advantage of the sequential nature of the snapshot
generation. Third, the DEIM basis can be efficiently updated.  See below for
more details regarding the last two points.

\subsubsection{Reduced storage costs} In model reduction techniques involving
dynamical systems, the snapshots used to compute the DEIM basis are generated
sequentially by a time-stepping method. When a fine-scale spatial
discretization is used, the number of degrees of freedom $n$ can be large and
the cost of storing many snapshots can be overwhelmingly large and even
prohibitively. We can take advantage of the sequential nature of the snapshot
generation to reduce storage costs.

Suppose that the target rank $r$ is known in advance (for simplicity, we do not
include oversampling).  The only step that involves manipulating the snapshots
is the computation of the sketch $Y = A\Omega$. Note that, we can alternatively
express this the sum of rank-1 outer products 
$$ Y = \sum_{j=1}^{n_s}A(:,j)\Omega(j,:)   .$$ 
Here $A(:,j)$ are the columns of the snapshot matrix,
and $\Omega(j,:)$ are the rows of $\Omega$.  This formula means that once each
snapshot $A(:,j)$ is generated, the sketch can be updated appropriately using
$A(:,j)\Omega(j,:)$, and then the snapshot can be discarded. If the target rank
$r$ is much smaller than the number of snapshots $n_s$, the storage cost is
lowered to $\mc{O}(nr) $ instead of $\mc{O}(nn_s)$ and these savings may be
substantial. An alternative option is maintaining two sets of sketches  as
advocated in~\cite{tropp2017practical}.

\subsubsection{Adapting the basis} 


Adapting the basis becomes necessary in certain
applications~\cite{peherstorfer2015online}; for example, in the offline stage,
the snapshot matrix $A$ maybe constructed with the objective of making the DEIM
approximation accurate over the entire parameter range, but computational
considerations may constrain the approximation to be accurate only over a certain
region in parameter space.  The DEIM
approximation may be used as a surrogate for the original function in an
optimization setting. As the optimization routine makes progress, the DEIM
approximation may not be  accurate if the optimization path deviates from the
region of accurate DEIM approximation; in this case, a good strategy may be to
update the DEIM basis based on the optimization trajectory. 

Randomized algorithms {allow} the user the flexibility to readily update the
DEIM basis by simply updating the sketch $Y$, instead of recomputing the 
SVD. Suppose the $j$-th column of the snapshot matrix that needs to be replaced
by $a_j$; we make the assumption that the target rank $r$ remains the same.
The corresponding sketch can be replaced as $Y \leftarrow Y
+(a_j-A(:,j))\Omega$; a thin-QR can be performed to obtain the basis and
the entire snapshot matrix is not necessary for adapting the basis. It is
easily seen how to simultaneously replace a block of columns.

\subsection{Improved accuracy via subspace iteration} For some applications,
the decay in the singular values may not be rapid enough to ensure that the
resulting subspace computed $\Wh$ is accurate. The basic idea is to replace the
sketch $Y = A\Omega$ in \cref{alg:basic} with the sketch $Y = (AA^\top)^q
A\Omega$. Essentially this means, running $q$ steps of subspace iteration where
$q$ is a non-negative integer. However, it is well known that a direct computation of the
sketch $Y = (AA^\top)^q A\Omega$ is numerically unstable and is significantly
affected by round-off error. To address this issue, numerically stable methods
alternate the matrix-vector products involving $A$ with a QR factorization.
\cref{alg:subspace} gives an idealized version of the algorithm that will be
suitable for analysis in \cref{ssec:acc_subspace}. For more details on a
numerically stable implementation,
see~\cite{halko2011finding,saad2011numerical}.

\LinesNumbered
\begin{algorithm2e}[!ht]
\DontPrintSemicolon
\SetKwInput{Input}{Input}
\SetKwInput{Output}{Output}
	\Input{ Snapshot matrix $A \in \R^{n\times n_s}$, target rank $r$, oversampling parameter $p \geq 1$ such that $r + p \leq \min\{n,n_s\}$, number of iterations $q \geq 0$. }
	\Output{Basis $\Wh$ with orthonormal columns.}
	Draw a standard Gaussian matrix $\Omega \in \R^{n_s\times (r+p)}$. \;
	Form $Y = (AA^\top)^qA\Omega$ and compute thin-QR factorization $Y = QR$. \;
	Form $B = Q^\top A$. Let $W$ be the left singular vectors of $B$ and set $W_r = W(:,1:r)$. \;
	Form $\Wh \leftarrow QW_r$. \;
\caption{Idealized randomized subspace iteration for range finding. Call as: $[\Wh]$ = RandSubspace$(A,r,p,q)$}
\label{alg:subspace}
\end{algorithm2e}

\section{Error Analysis}\label{sec:error} In~\cref{ssec:accuracy}, we derive
bounds for the accuracy of the DEIM approximation, when a perturbed DEIM basis
$\Wh$ is used instead of the standard DEIM basis $W$. The bounds are applicable
whether $\Wh$ is obtained using a randomized algorithm or using any other
approximation algorithm. For example, due to the inexactness in the function
evaluations,  the snapshot matrix $A$ may be perturbed by $\Delta A$; In this
case, the left singular vectors $\Wh$ of $A + \Delta A$ can be used as the
perturbed DEIM basis. The results in this subsection are applicable to this
setting as well. In~\cref{ssec:acc_subspace}, we derive bounds for the angles
between the subspaces spanned by columns of $W$ and $\Wh$, when $\Wh$ is
obtained using the randomized range finding algorithms. We then use these
bounds to fully quantify the error in the R-DEIM approximation.

\subsection{Notation and canonical angles} \label{ssec:not} Let $W \in
\R^{n\times r}$ be the standard DEIM basis, obtained from the first $r$ left
singular vectors of $A$ and let $S \in \R^{n\times s}$ be the selection
operator. For generality, we will assume that the selection operator contains
columns from the identity matrix, but maybe scaled (see \cref{sec:point} for
examples).  The following orthogonal projectors will be of use in what follows
$$P_W \equiv  WW^\top \qquad P_S  \equiv  SS^\dagger . $$ 
We note that
\cref{eqn:obl0,eqn:obl1} still hold if $s \geq r$ and $\rank(S^\top W) = r$,
and all three relations hold if $s=r$,  and $\rank(S^\top W) = r$.

To distinguish from the standard DEIM basis, denote $\Wh \in \R^{n\times r}$ {as} the ``perturbed''
basis (obtained, for example, from \cref{alg:basic}) \changes{with orthonormal columns} and the corresponding
selection operator $\Sh \in \R^{n\times s}$; assume that $\rank(\Sh^\top \Wh) = r$. Define the corresponding
projectors 
$$P_{\Wh} \equiv \Wh\Wh^\top, \qquad P_{\Sh} = \Sh\Sh^\dagger, \qquad \Dh \equiv \Wh (\Sh^\top \Wh)^{\dagger} \Sh^\top.$$ 
We allow for the possibility that the selection operator $\Sh$ and $S$ maybe
different; these maybe computed based on the perturbed DEIM basis $\Wh$ and the standard
DEIM basis $W$ respectively.  Throughout this section, we assume that
the DEIM projectors $\D,\Dh$ do not equal either the zero matrix or the
identity matrix.

The analysis of the error in the DEIM approximation requires computing the
overlap between two subspaces of $\R^n$ which, in turn, can be described in
terms of canonical angles. We now briefly review some definitions and
properties of canonical angles\changes{; see~\cite[Chapter I.5]{stewart1990matrix}}. Denote {the} range spaces of $W$ and $\Wh$, by
$\mc{W}$ and $\mch{W}$ respectively; both these subspaces have dimension $r$.
The principal or canonical angles, between $\mc{W}$ and $\mch{W}$ are
$\theta_1, \dots, \theta_r = \theta_{\max} $ and satisfy
\[ 0 \leq \theta_1 \leq \dots \leq \theta_{\max} \leq \pi/2. \] 
The canonical angles can be computed using the SVD.  Denote the singular values
$\sigma_1 \geq \dots \geq \sigma_r$ of $W^\top \Wh$;  then the canonical
angles are $\theta_i = \arccos(\sigma_i)$ and furthermore,
\begin{equation}\label{eqn:thetamax}  \sin\theta_{\max} = \| P_W - P_{\Wh}\|_2 = \|(I_n-P_W)P_{\Wh} \|_2 = \|(I_n-P_{\Wh})P_{W} \|_2. \end{equation}
Similarly, let $\psi_{\max}$ denote the largest canonical angle between the
pair of subspaces $\range(S)$ and $\range(\Sh)$. {Equalities analogous to~\cref{eqn:thetamax} also hold for $\psi_{\max}$.}

\subsection{DEIM approximation}\label{ssec:accuracy}
We present two theorems that quantify the error in the perturbed DEIM
approximation. We use the notation that was established in \cref{ssec:not}. We
{remind the reader} that both the basis $\Wh$ and the
selection operator $\Sh$ may be different than the standard DEIM basis $W$ and
selection operator $S$ respectively.\begin{theorem}\label{thm:error1} Let $ \Dh f$ be the perturbed DEIM approximation to $\D f$ and assume that  $s \geq r$ \changes{and} $\rank(\Dh) = r$. The approximation error satisfies 
\[ \| f - \Dh f \|_2 \leq  \| \Dh \|_2 \left( \|(I_n-P_W)f\|_2 + \sin \theta_{\max} \|P_W f\|_2 \right),  \]
\end{theorem}
\begin{proof}
 From the expression $f - \Dh f =  (I_n-\Dh) (P_W + I_n-P_W)f$, and by triangle inequality 
\[\| f - \Dh f\|_2 \leq  \| (I_n-\Dh) P_W f\|_2 +  \|(I_n -\Dh)(I_n-P_{W})f\|_2. \]
 Recall  that \changes{since $s \geq r$ and $\rank(\Dh) = \rank(\Sh^\top \Wh) = r$; therefore, } $I_n -\Dh =  ( I_n - \Dh)(I_n-P_{\Wh})$. By assumption {$\Dh \neq 0$  and $\Dh \neq I_n$} and therefore, by \cref{oblique}  $\|I_n-\Dh\|_2 = \|\Dh\|_2$. Using these identities,  we get 
\[ \| f - \Dh f\|_2 \leq \|\Dh\|_2 \left(\| (I_n-P_{\Wh})P_Wf \|_2 + \|(I_n-P_{W})f\|_2 \right).\] 
Use submultiplicativity 
$$\|  (I_n-P_{\Wh})P_Wf\|_2  \leq \|(I_n-P_{\Wh})P_W\|_2 \|P_Wf\|_2 = \sin \theta_{\max} \|P_W f\|_2, $$
and plug this into the previous equation to obtain the advertised result.
\end{proof}

The first term is similar to the error in the DEIM approximation \cref{l_deim}.
The second term is the additional error introduced by using the perturbed DEIM
basis $\Wh$ and is quantified by sine of the largest canonical angle between
the subspaces $\mc{W}$ and $\mc{\Wh}$---if these subspaces are identical, and
if $S$ equals $\Sh$, then the error reduces to the standard DEIM approximation. 


It is worth comparing this error bound with that of the standard DEIM approximation \cref{l_deim}. A little bit of algebra reveals that (if $P_W f \neq 0$) 
\[ \| f - \Dh f\|_2 \leq  \kappa  \|(I_n-P_W)f\|_2 ,  \]
where $\kappa$ is an amplification factor, and 
\[ \kappa = \left( 1 + \frac{ \sin\theta_{\max} }{ \|(I_n-P_W)f\|_2/ \| P_Wf\|_2 }\right) \|\Dh\|_2.\]
An important distinction between the original DEIM result and \cref{thm:error1} is that the condition number now appears to explicitly depend on the function $f$.

A different proof technique leads to a qualitatively different bound that includes both angles $\theta_{\max}$ and $\psi_{\max}$. It requires the additional assumption that the number of selected indices $s$ equals the dimension of the DEIM basis $r$.
\begin{theorem}\label{thm:error2} Let $ \Dh f$ be the perturbed DEIM approximation to $\D f$. Assume that $s =r$ and $\rank(\D) = \rank(\Dh ) =r $.  Then, the approximation error is 
\begin{equation} \begin{split}
\| f - \Dh f\|_2 \leq  & \>  \| \D \|_2  \|(I_n-P_W)f\|_2 \\
	&  +  \|\D\|_2 \|\Dh\|_2 \left(   \sin\psi_{\max} \|(I_n-P_W)f\|_2 + \sin\theta_{\max} \|P_Sf\|_2  \right) . \end{split}  
\end{equation}
\end{theorem}
\begin{proof}
The proof uses the decomposition
 \[ I_n-\Dh = I_n - \D + (I_n-\Dh)\D -\Dh(I_n-\D).\]
 Applying triangle inequality 
 \[ \| {(I_n- \Dh) f}\|_2 \leq \|(I_n-\D)f\|_2 +\|(I_n-\Dh)\D f\|_2 + \|\Dh(I_n-\D) f\|_2.\]
 The first term is the standard DEIM error and is bounded by $\|\D\|_2 \|(I_n-P_W)f\|_2$. For the subsequent terms, using the analogues of~\cref{eqn:obl0,eqn:obl1,eqn:obl2}, we have the equalities 
\begin{equation*}
 \begin{aligned}
(I_n-\Dh)\D = & \> (I_n-\Dh)(I_n-P_{\Wh})P_W\D P_S \\ 
\Dh(I_n-\D) = & \> \Dh P_{\Sh}(I_n-P_S)(I_n-\D)(I_n-P_W).
\end{aligned}
\end{equation*}
Therefore, with repeated application of the submultiplicativity inequality 
\[\|(I_n-\Dh)\D f\|_2 = \|(I_n-\Dh)(I_n-P_{\Wh})P_W\D P_Sf\|_2 \leq \sin \theta_{\max} \|I_n-\Dh\|_2 \|\D\|_2 \|P_Sf\|_2. \] 
The identity \cref{oblique}, along with the assumption {$\Dh\neq 0$ and $\Dh \neq I_n$}, completes the second term. The last term is obtained in a similar manner. 
\end{proof}

The interpretation of this theorem {is similar} to that of \cref{thm:error1}. If
$\range(S) = \range(\Sh)$ and $\range(W) = \range(\Wh)$, then the two trailing
terms drop out and we are left with the DEIM error \cref{l_deim}. If $\mc{W} = \mc{\Wh}$ but
$\range(S) \neq \range(\Sh)$ then $\sin\theta_{\max} = 0$. Conversely, if
$\mc{W} \neq \mc{\Wh}$ but $\range(S) = \range(\Sh)$ then $\sin\psi_{\max} =
0$.

Which bound is better? Our analysis in \cref{thm:error2}{, which uses}
the perturbation results of oblique projectors, it is likely to be sub-optimal.
Note that in the second expression, we have the multiplicative factor $\|\D\|_2
\|\Dh\|_2$; since both terms are at least $1$, this expression can be large and
clearly undesirable. {This is what we see in \cref{fig:rdeimbound}.}
\begin{figure}[!ht]
\centering
\includegraphics[scale=0.4]{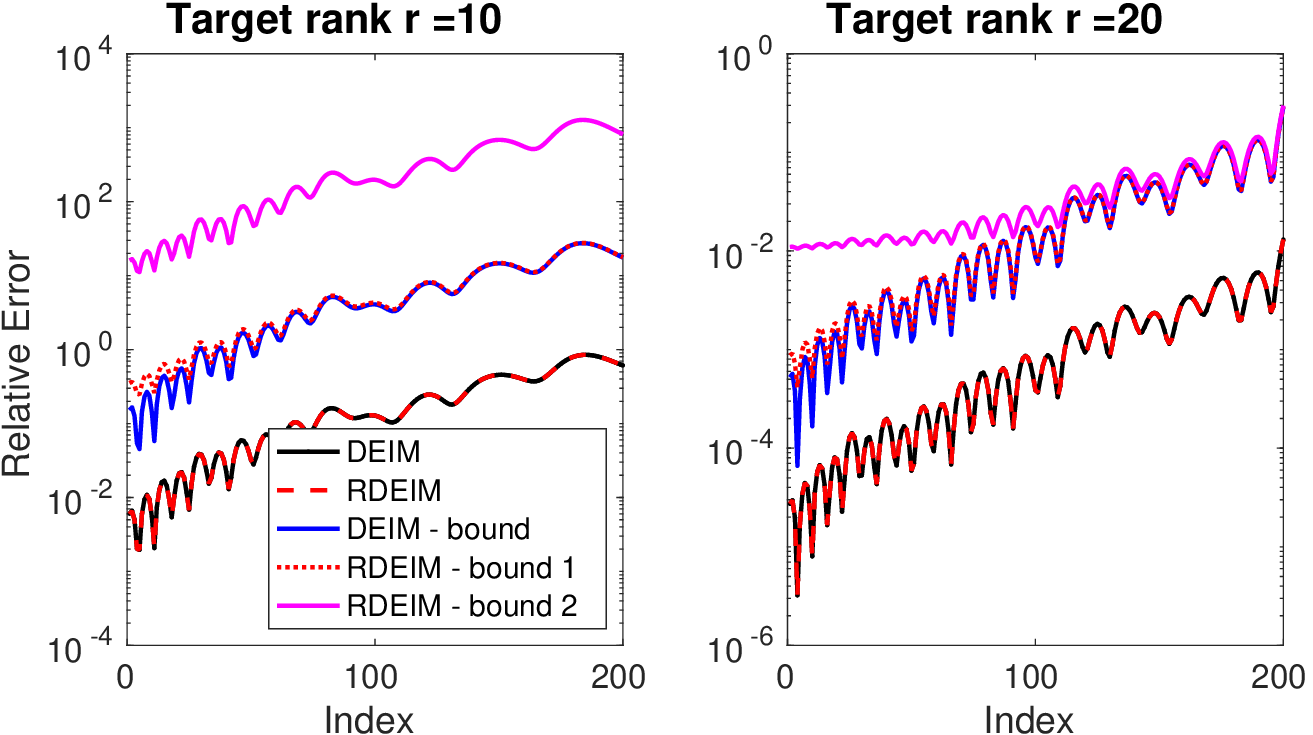}
\caption{{Comparison of the R-DEIM error and the bounds in~\cref{thm:error1} (Bound 1) and~\cref{thm:error2} (Bound 2). For comparison we also include the DEIM error and the bound \cref{l_deim}.}}
\label{fig:rdeimbound}
\end{figure}

\paragraph{Illustration of the bounds} 

This example is based on~\cite[Example 3.1]{drmac2016new}. Let 
\begin{equation}\label{e_func_ex1}
 {f}(t;\mu) = 10\exp(-\mu t) \left(\cos(4\mu t) + \sin(4\mu t)\right) \qquad 1 \leq t \leq 6, 0 \leq \mu \leq \pi. 
\end{equation}

The snapshot set is generated by taking $n_\mu = 100$ evenly spaced values of $\mu$ and
$n=10,000$ evenly spaced points in time. The thin SVD of this matrix is
computed and the left singular vectors corresponding to the first $34$ modes
are used to define $W$. We use \cref{alg:rangetwo} to obtain a randomized
basis $\Wh$ and we use an oversampling parameter $p=20$. We
consider two different target ranks $r=10,20$. To report the error we define
the vectors $f_j = \bmat{f(t_1;\mu_j)&\dots& f(t_n;\mu_j)}^\top$, for
$j=1,\dots,n_\mu$, and the relative error defined as 
\[ \text{Rel Err}(\mu_j) \> \equiv \> \frac{ \|f_j - \D f_j\|_2}{ \| f_j\|_2} \qquad j = 1,\dots,n_\mu.\] 
The results of the comparison are provided in \cref{fig:rdeimbound}. We also
plot the bounds for the DEIM approximation \cref{l_deim} and the R-DEIM
approximation\cref{thm:error1}. The point selection was done using the pivoted
QR algorithm~\cite{drmac2016new} (see \cref{ssec:det}). We see that the error
using the R-DEIM approximation closely follows the error of the DEIM algorithm.
Although both the DEIM and the R-DEIM bounds over-predict the error, we see
that the bound for R-DEIM is in close agreement with the bound for DEIM. We did
not plot the bound from \cref{thm:error2}, because we did not find it to be
very accurate.   

\subsection{Accuracy of the subspaces}\label{ssec:acc_subspace} The previous
subsection reveals that the accuracy of R-DEIM depends on the largest canonical angle
between the subspaces $\mc{W}$ and $\mch{W}$ respectively. In this subsection,
we {shed more light} into the accuracy of this quantity.

Assume that the snapshot matrix $A$, with singular vectors $W = W_1$, is
perturbed to $\Ah$.  The perturbation may be either deterministic or random.
Bounds for the canonical angles can be obtained from results known in
the literature as ``sin theta'' theorem. Below is one example of such a
theorem.
\begin{lemma}\label{lemma:angle2}
	Let $\Ah \in \R^{m \times n}$ be a perturbation of $A\in \R^{m \times n}$ such that  $\sigma_{r}(A) - \sigma_{r+1}(\Ah) > 0$ and denote the left singular vectors of $\Ah$ by $\Wh$. Then 
\[ \sin\theta_{\max} \leq \frac{\max\{\|(A - \Ah)\widehat{Z}_1\|_2, \|(A^\top - \Ah^\top)\widehat{W}_1\|\}}{\sigma_{r} - \sigma_{r+1}(\Ah)} .\]
\end{lemma}
\begin{proof}This follows from~\cite[equation (4.15)]{wedin1983angles}.
\end{proof}
For some applications it may also be more convenient to bound the numerator with $\|A-\Ah\|_2$.

When $\Wh$ is computed using \cref{alg:subspace}, these results can be made more precise. 
\begin{theorem}\label{thm:expect} Let $\Wh \in \R^{n\times r}$ be the DEIM basis obtained using \cref{alg:subspace} and let $p\geq 2$. Assume that \changes{$1 \leq r < \rank(A)$} and the singular value ratio $\gamma = \sigma_{r+1}/\sigma_r < 1$, so that the subspace is well-defined. Let the constant $C$ be defined as 
\[ C \equiv \sqrt{\frac{r}{p-1}} + \frac{e\sqrt{(r+p)(n_s-r)}}{p} \,. \]
Then
\[ \mb{E}\, \sin\theta_{\max} \leq  \frac{\gamma^{2q+1} C}{1-\gamma}.\]
\end{theorem}
\begin{proof}
See \cite[Theorem 4]{saibaba2018randomized}.
\end{proof}

The interpretation of this theorem is that the largest canonical angle
converges to $0$ as the number of subspace iterations increase. More precisely,
if we want $\sin\theta_\text{max}$ to be bounded, in expectation, by some
positive parameter $\varepsilon < 1$, then the number iterations (depending on
$\varepsilon$) should satisfy
\[ q_\varepsilon \geq \frac12 \left( \frac{\log \varepsilon (1-\gamma)/(\gamma C)}{\log \gamma}\right). \]

This result also shows that the accuracy of the R-DEIM approximation depends on
the singular value ratio $\gamma$---smaller this ratio, more accurate the subspace. When
$\gamma \approx 1$ the bound in \cref{thm:expect} is devastating, but is to be
anticipated since this means that the subspace may be poorly defined.

\subsection{Randomized DEIM basis with deterministic point
selection}\label{ssec:det} We briefly review the various choice of the
selection operator $S$ and review the error bounds associated with each choice.
We refer to the condition number $\|\D\|_2$ as the DEIM error constant. We then
show how to combine the R-DEIM basis with existing point selection techniques.

 In~\cite{chaturantabut2010nonlinear}, a greedy approach (which we call DEIM
selection algorithm) was used to determine the point selection indices. For
this algorithm, the DEIM error constant is bounded by $\mc{O}(\sqrt{n})^r$.
However, numerical experiments showed that the bound for the DEIM error
constant was pessimistic. Subsequent analysis in~\cite{sorensen2016deim}
improved this bound to $\mc{O}(\sqrt{nr}2^r)$ and constructed an explicit
matrix for which this bound could be attained asymptotically. They concluded
that while the bound was large, numerical experiments showed that the point
selection algorithm worked quite well, in practice.  Recent work
in~\cite{drmac2016new} developed a different approach which used pivoted QR
(PQR) on $W^\top$ to obtain the selection operator $S$.  As with the DEIM
selection algorithm, the error constant can be large and could be attained by
specially constructed adversarial cases. Numerical experiments suggest that the
performance is comparable to the DEIM selection algorithm and is often better.
More recently, the authors~\cite{drmac2017discrete} used the Gu-Eisenstat
strong rank-revealing QR (sRRQR) algorithm~\cite{gu1996efficient} to obtain the bound
\[ \|\D\|_2 \leq   \sqrt{1 + \eta^2r(n-r)} \equiv D_\text{sRRQR},\] 
where $\eta \geq 1$ is
a user-specified parameter (the authors in~\cite{gu1996efficient} call this
parameter $f$). The DEIM error constant using sRRQR is significantly lower than
the DEIM selection algorithm or PQR algorithm{; see \cref{tab:summary} for the corresponding DEIM error constants.}  The cost of sRRQR is roughly \changes{$\mc{O}(nr^2)$ ignoring logarithmic factors~\cite{gu1996efficient}. }
In
numerical experiments, the performance of sRRQR is similar to that of PQR;
therefore, we use the latter in our experiments because of its lower
computational cost. 

\LinesNumbered
\begin{algorithm2e}[!ht]
\DontPrintSemicolon
\SetKwInput{Input}{Input}
\SetKwInput{Output}{Output}
	\Input{ Snapshot matrix $A \in \R^{n\times n_s}$, target rank $r$,  oversampling parameter $p \geq 1$, number of iterations $q \geq 0$.  User-defined parameter $\eta \geq 1$. }
	\Output{Basis $\Wh \in \R^{n\times r}$ with orthonormal columns and selection operator $\Sh \in \R^{n\times r}$ defining the R-DEIM projector $\Dh = \Wh (\Sh^\top W)^\dagger \Sh^\top$.}
	Construct $\Wh$ as $[\Wh]$ = RandSubspace$(A,r,p,q)$.\;
	Apply sRRQR algorithm with parameter $\eta \geq 1$ to obtain 
 \[ \Wh^\top \bmat{\Pi_1  & \Pi_2} = Q \bmat{R_{11} & R_{12}}.\] 
Set $\Sh = \Pi_1 \changes{\in \R^{n\times r}}$. \;
\caption{Randomized DEIM approximation with sRRQR point selection.}
\label{alg:randsrrqr}
\end{algorithm2e}

Suppose the DEIM basis $\Wh$ is generated using the randomized algorithm
\cref{alg:subspace}. We can apply sRRQR to $\Wh^\top$ to obtain
\[ \Wh^\top \bmat{\Pi_1 & \Pi_2} = Q\bmat{R_{11} & R_{12}}.\]
Here $\bmat{\Pi_1 & \Pi_2} \in \R^{n\times n}$ is a permutation matrix, $Q \in
\R^{r\times r}$ is orthogonal, and $R_{11} \in \R^{r\times r}$ is upper triangular with positive
diagonals. The selection operator $S$ is taken to be $\Pi_1 \in \R^{n\times r}$, which selects
well conditioned rows of $\Wh$. This is summarized in \cref{alg:randsrrqr}. The
following theorem analyzes the error in the resulting approximation. 
\begin{theorem}
Let $\Wh$ and \changes{$\Sh$} be the outputs of \cref{alg:randsrrqr} and define the
R-DEIM approximation $$\rdeim = \Wh(\Sh^\top \Wh)^\dagger \Sh^\top f.$$ With the
assumptions and notation of \cref{thm:expect}, the expected error in the R-DEIM
approximation  satisfies   
\[ \mb{E}_\Omega \| f - \rdeim \|_2 \leq D_\text{sRRQR} \left( \|(I_n-P_W)f\|_2 + \frac{\gamma^{2q+1} C}{1-\gamma}\|P_Wf\|_2 \right). \]
\end{theorem}
\begin{proof}
By~\cite[Lemma 2.1]{drmac2017discrete} we obtain $\|\Dh\|_2 = \| (\Sh^\top
\Wh)^{-1}\|_2 \leq D_\text{sRRQR}$. This ensures that $\rank(\Dh) =
\rank(\Sh^\top \Wh) = r$\changes{. The assumption $1 \leq r < \rank(A)$ ensures $\Dh \neq 0$ or $\Dh \neq I_n $;} therefore, we can apply \cref{thm:error1} to obtain
\[ \|f-\rdeim\|_2 \leq D_\text{sRRQR} \left( \|(I-P_W)f\|_2 + \sin\theta_{\max} \|P_W f\|_2 \right).\] 
We apply the result from \cref{thm:expect}
to the above equation, which completes the proof. 
\end{proof}

\section{Randomized Point selection}\label{sec:point} To our knowledge, the
best known bounds for the DEIM error constant are obtained using the sRRQR
algorithm; however, as mentioned earlier, the computational cost of the sRRQR
algorithm can be high and maybe prohibitively expensive for applications of
interest. To tackle this computational challenge, sampling based randomized
approaches have been previously
proposed~\cite{peherstorfer2015online,drmac2016new,peherstorfer2018stabilizing}. 

The sampling approach for point selection, randomly samples $s\geq r$ indices
from the index set $\{1,\dots,n\}$, according to a pre-specified discrete
probability distribution, to determine the selection operator $S$. The
computational cost of point selection by sampling is $\mc{O}(ns)$; in
comparison, both the DEIM selection and PQR algorithms cost $\mc{O}(nr^2)$.
Sampling based techniques is that they are readily parallelizable and
therefore, advantageous for large-scale problems. There are two competing
issues to consider when deciding between sampling strategies: the DEIM error
constant, and the number of required samples.  A low DEIM error constant is
desirable but may require many samples $s$, which increases the computational
cost.  An example of sampling strategy includes uniform sampling, either with
or without replacement. When the DEIM basis $W$ has high coherence (for a definition, see
\cref{eqn:leverage} and the discussion below it), the number of samples
required $s$ can be large to ensure a small DEIM error constant. Therefore, we
do not use uniform sampling for the point selection. See \cite{ipsen2014effect}
for additional discussion on the effect of coherence on sampling from matrices
with orthonormal columns.  In this section, we propose randomized algorithms
for the selection operator $S$ and develop bounds for the proposed selection
operators.

We propose two different randomized point selection techniques. The first
method is based on leverage scores (see \cref{eqn:leverage} for a definition). 
While the point selection stage is computationally efficient, the overall cost can be high since 
$s \sim \mc{O}(r\log r)$ samples need to be drawn---this also corresponds to
the number of points at which the nonlinear function $f$ is evaluated. In
certain applications, it is desirable to pick only $r$ samples. To address this
issue, we propose a hybrid point selection algorithm which retains the
computational advantages of sampling-based point selection, but only samples 
$r$ indices, thereby retaining the favorable properties of deterministic
methods.

\subsection{Randomized point selection with standard DEIM basis} We first
develop algorithms for randomized point selection with the standard DEIM basis
$W\in \R^{n\times r}$. The analysis can be extended to the randomized DEIM
basis and is considered in the next subsection.

The {\em leverage scores} of $W$ are defined to be
\begin{equation}\label{eqn:leverage} 
\ell_j \> \define \> \normtwo{\vec{e}_j^\top\mat{W}}^2, \qquad j=1,\dots,n,
\end{equation} 
where $\vec{e}_j$ is the $j$-th column of an $n\times n$ identity matrix.
Equivalently, the leverage scores are the squared row norms of $W$, or  
alternatively, the diagonals of the projector $P_W = WW^\top$. The largest
leverage score is known as the {\em coherence} and the sum of the leverage scores
satisfies $\sum_{j=1}^n \ell_j = r$. Based on the leverage scores, we can
define the following discrete probability mass function (pmf) $\pi_j = \ell_j /r$ for
$j=1,\dots,n$. In what follows, we instead use the related pmf 
\begin{equation}
\label{eqn:prob}
\pi_j^\beta =   \frac{\beta\ell_j}{r} + \frac{(1-\beta)}{n}, \qquad j=1,\dots,n. 
\end{equation}
Here $ 0 < \beta < 1$ is a user-defined constant, and the modified pmf is a convex
combination of the  leverage score pmf  and the uniform pmf. This modified pmf
is beneficial since it can handle rows with zero leverage scores.

\paragraph{Leverage score point selection approach} The leverage score approach
constructs a sampling matrix by selecting indices $\{t_1,\dots,t_s\}$,
independently and with replacement, from the index set $\{1,\dots,n\}$, with
probabilities given by \cref{eqn:prob}. We construct the selection operator $S$
as follows: the $j$-th column of $S$ is $e_{t_j}/\sqrt{s\pi_{t_j}^\beta}$, for
$j=1,\dots,s$. This is summarized in \cref{alg:lev}. These columns are scaled
in this manner to ensure that $SS^\top$ equals the identity matrix $I_n$ in
expectation; that is, $SS^\top$ is an unbiased estimator of the identity
matrix. 
\begin{lemma} 
Let $W\in \R^{n\times r}$ be a matrix with orthonormal columns. Let the selection operator be constructed as described in \cref{alg:lev}. Then $\mb{E}[SS^\top] = I_n$.
\end{lemma}
\begin{proof}
We write $SS^\top$ as the outer product representation
\[ SS^\top = \sum_{j=1}^{s} Y_j \qquad Y_j \equiv \frac{1}{s\pi_{t_j}^\beta}e_{t_j}e_{t_j}^\top. \]
It is easy to verify that $\mb{E}[Y_j] = I_n/s$. By the linearity of expectations, it follows that $\mb{E}[SS^\top] = I_n.$
\end{proof}

The number of selected indices $s$ is chosen to be
\begin{equation}
\label{eqn:cls} \cls \equiv  \Bigl\lceil\frac{2r}{\beta\epsilon^2} \log (r/\delta) \Bigr\rceil,
\end{equation}
where $\lceil \cdot \rceil$ is the ceiling function. 
This choice will be justified in \cref{thm:sampling}. 

\LinesNumbered
\begin{algorithm2e}[!ht]
\DontPrintSemicolon
\SetKwInput{Input}{Input}
\SetKwInput{Output}{Output}
\Input{ Matrix $W \in \R^{n\times r}$ with orthonormal columns,  number of samples $s\geq r$. Probabilities $\{\pi_j^\beta\}_{j=1}^n$ defined in \eqref{eqn:prob}.}
\Output{Matrix $S \in \R^{n\times \cls}$}
	\For {$j=1,\dots,s$} {
		Select index $t_j$, independently and with replacement, from  $\{1,\dots,n\}$ with probabilities $\{\pi_j^\beta\}_{j=1}^n$.\; 
		Set $S(:,j) = \frac{1}{\sqrt{s\pi_{t_j}^\beta}} e_{t_j}$. \;
	}
\caption{Leverage score point selection. Call as: $[S]$ =  LeverageScorePS$(W,\{\pi_j^\beta\}_{j=1}^n,s)$}
\label{alg:lev}
\end{algorithm2e}

 \paragraph{Hybrid point selection approach} The hybrid approach we propose has
two stages: a randomized point selection stage, which uses the leverage score
distribution to select $\cls$ indices, and a deterministic approach, which uses
the sRRQR algorithm to choose $r$ indices out of $\cls$. 
\begin{description}
\item[1. Randomized stage] In the first stage, we use leverage score approach
(\cref{alg:lev}) with $s =\cls$ and denote the point selection matrix $S_1 \in
\R^{n\times \cls}$. The resulting matrix $S_1^\top W$ extracts $\cls$ rows from
$W$, with appropriate scaling.  
\item [2. Deterministic stage] In the second
stage, we apply sRRQR to $W^\top S_1$ to select exactly $r$ rows from the
matrix $S_1^\top W$. Let $S_2 \in \R^{\cls \times r}$ denote the point
selection matrix obtained using sRRQR.  
\end{description}
  Denote this
composite selection matrix as $S = S_1S_2 \in \R^{n\times r}$ which consists of
columns from $n\times n$ identity matrix that are scaled by the appropriate
factors. Similar algorithms have been proposed in
\cite{boutsidis2009improved,broadbent2010subset}. However, the specific choice
of the sampling distribution and the subsequent analysis are different.

\LinesNumbered
\begin{algorithm2e}[!ht]
\DontPrintSemicolon
\SetKwInput{Input}{Input}
\SetKwInput{Output}{Output}
\Input{ Matrix $W \in \R^{n\times r}$ with orthonormal columns, number of samples $\cls$, parameter $\eta \geq 1$. \\ Probabilities $\{\pi_j^\beta\}_{j=1}^n$ defined in \eqref{eqn:prob}. }
\Output{ Matrices $S_1 \in \R^{n\times \cls}$ and $S_2 \in \R^{\cls \times r}$ that define $S= S_1S_2$.}
 \tcc{Stage 1. Randomized Stage } 
 [$S_1$]  = LeverageScorePS$(W,\{\pi_j^\beta\}_{j=1}^n,\cls)$. \; 
  \tcc{Stage 2. Deterministic Stage}  
 Perform sRRQR with parameter $\eta$ on $W^\top S_1$. $ W^\top S_1 \bmat{\Pi_1 & \Pi_2} = Q \bmat{R_{11} & R_{12}}$ . \; 
 Set $S_2 = \Pi_1$.\;
\caption{Hybrid Point selection.}
\label{alg:hyb}
\end{algorithm2e}
In the analysis of the DEIM approximation \cref{l_deim}, the condition number is determined by the DEIM error constant $\|\D\|_2$. We derive bounds for this constant when the selection operator is obtained using the leverage score and the hybrid approaches.	
\begin{theorem}\label{thm:sampling} Let $W \in \R^{n\times r}$ be a fixed matrix with orthonormal columns and let $0 < \epsilon, \delta < 1$ be user-defined parameters. Let $S_1 \in \R^{n\times \cls}$ and $S = S_1S_2\in \R^{n\times r}$ be the outputs of~\cref{alg:lev,alg:hyb}, with the number of samples $s= \cls$, and define the corresponding DEIM operators $ \D_{\text{LS}} \equiv W(S_1^\top W)^\dagger S_1^\top$ and $\D_\text{Hy} \equiv W(S^\top W)^\dagger S^\top$.  With probability at least $1-\delta$, the DEIM error constant for
\begin{itemize} 
\item  the leverage score approach satisfies 
\[ \| \D_{\text{LS}}\|_2  \leq  \sqrt{\frac{n/\cls}{(1-\beta)(1-\epsilon)}} \equiv  D_\text{LS} .\]
\item  the hybrid approach satisfies 
\[ \| \D_{\text{Hy}}\|_2  \leq  D_\text{LS}\sqrt{1+\eta^2r(\cls-r) } \equiv  D_\text{Hy}.\]
\end{itemize}
\end{theorem}

For both the leverage score and the hybrid point selection approaches, the point selection operator $S$ contains (appropriately scaled) columns from the identity matrix. A few differences between the standard DEIM approach are worth pointing out \changes{(assume that $\D \neq 0$ and $\D \neq I_n$)}: 
\begin{enumerate}
\item The matrix $S$ no longer has orthonormal columns; therefore, 
\[ \|\D\|_2 = \| I_n - \D\|_2 = \| (S^\top W)^\dagger S^\top\|_2.\]
The first equality holds because $\D$ is an oblique projector, see \cref{oblique}. 
\item Second, the DEIM implementation has to be altered appropriately. There are two steps: first, the components of $f(\cdot)$ as determined by $S$ are extracted, and second, these components are scaled by the corresponding scaling factor. 
\item We can combine \cref{thm:sampling} along with \cref{l_deim} to derive the error in the DEIM approximation. With the assumption and notation of \cref{thm:sampling}, the following bounds hold with probability at least $1-\delta$
\[ \|f-\D_\text{LS}f\|_2 \leq D_\text{LS} \|(I_n-P_W)f\|_2 \qquad \|f-\D_\text{Hy}f\|_2 \leq D_\text{Hy} \|(I_n-P_W)f\|_2.\]
This is obtained by combining \cref{thm:sampling} with \cref{l_deim}.
\end{enumerate}

\begin{proof}[\cref{thm:sampling}]
The DEIM operators satisfy the inequalities
\[ \|\D_\text{LS} \|_2 = \|W (S_1^\top W)^{\dagger} S_1^\top \|_2 \leq \| (S_1^\top W)^{\dagger} \|_2 \|S_1\|_2\]
and
\[  \|\D_\text{Hy} \|_2 = \|W (S^\top W)^{\dagger} S^\top \|_2 \leq \| (S^\top W)^{\dagger} \|_2 \|S\|_2.\]

\paragraph{Leverage scores approach}{We have to find upper bounds for $\|S_1\|_2$ and $\| (S_1^\top W)^{\dagger} \|_2$. To bound $\|(W^\top S)^\dagger\|_2$, we first observe that 
\[ \pi_j^\beta \geq \frac{\beta\ell_j}{r}.\]
If we take the number of columns of $S_1$ taken to be $ s =\cls$, the operator $S_1$ satisfies the conditions of \cite[Theorem 6.2]{holodnak2015randomized}. It follows from this theorem that with probability at least $1-\delta$
\begin{equation}\label{eqn:s1w}  1 - \epsilon \leq \sigma_r^2(S_1^\top W) \qquad \text{or} \quad \|(S_1^\top W)^\dagger\|_2 \leq \frac{1}{\sqrt{1-\epsilon}}. \end{equation}
 
We now bound $\|S_1\|_2$. If $\beta = 1$, that is the pmf only contains contributions from the leverage scores. The norm $\|S_1\|_2$ may be unbounded if zero leverage scores are encountered. However, if $\beta \neq 0$, since $\pi_j^\beta \geq \frac{1-\beta}{n}$
\begin{equation} \label{eqn:s1}\|S_1\|_2 \leq \max_{1 \leq j \leq n} \sqrt{\frac{1}{\cls\pi_j^\beta}} \leq \sqrt{\frac{n}{\cls(1-\beta)}} .\end{equation}
Putting this together, we get 
\[ \|\D_\text{LS} \|_2 \leq \| (S_1^\top W)^{\dagger} \|_2 \|S_1\|_2 \leq \sqrt{\frac{n/\cls}{(1-\epsilon)(1-\beta)}}.\]
\paragraph{Hybrid approach} Similar to the previous part of the proof, we have to bound $\| (S^\top W)^{\dagger} \|_2$  and $\|S\|_2$. Applying sRRQR to $(W^\top S_1)$ gives 
\[  (W^\top S_1) \bmat{\Pi_1 & \Pi_2} = Q\bmat{R_{11} & R_{12}}, \]
where $Q \in \R^{r \times r}$ is an orthogonal matrix; $R_{11} \in \R^{r\times
r}$ is upper triangular; and  $\bmat{\Pi_1 &\Pi_2}  \in
\R^{C_\text{LS} \times \cls}$ is a
permutation matrix with $\Pi_1 \in \R^{\cls\times r}$. Recall that $S_2 = \Pi_1$ and that $S = S_1 S_2$. Then, the singular value 
bounds~\cite[Lemma 3.1]{gu1996efficient} ensure that 
\[ \frac{\sigma_r(S_2^\top W)}{\sqrt{1+\eta^2 r(\cls-r)} }  \leq \sigma_r(S^\top W).\] 
Therefore, 
\[ \|(S^\top W)^{-1}\|_2 \
	\leq \sqrt{1+\eta^2 r(\cls-r)} 
	\| (S_1^\top W)^\dagger \|_2 \leq \sqrt{\frac{1+\eta^2 r(\cls-r) }{1-\epsilon}}. \] 
The bound for $\| (S_1^\top W)^\dagger \|_2$ follows from \cref{eqn:s1w}. To finish the proof, it remains to bound $\|S\|_2  \leq \|S_1\|_2 \|S_2\|_2 $. Since $S_2$ contains columns from the identity matrix, $\|S_2\|_2 = 1$ and we have an upper bound for $\|S_1\|_2$ in \cref{eqn:s1}. Combining all the intermediate steps, the bound for $\D_\text{Hy}$ then follows readily.}
\end{proof}

To shed light on the DEIM error constants, we give some representative values. Suppose $\beta = 1/2$, $\eta = 2$ and $\epsilon = 9/10$, the number of samples required are $$\cls \leq 5r\log (r/\delta)$$ and 
\[ D_\text{LS} = \sqrt{\frac{20n}{\cls}} \qquad D_\text{Hy} \leq  D_\text{LS} \sqrt{1 + 20 r^2{\log (r/\delta)}}.\]

In terms of asymptotic complexity, the DEIM error constant for the leverage
score algorithm is $\mc{O}(\sqrt{n/\cls})$ whereas for the hybrid algorithm, it is
$\mc{O}(\sqrt{nr})$. This is to be compared with the sRRQR algorithm for which
the DEIM error constant is $\mc{O}(\sqrt{nr})$. In terms of computational
costs, the cost of computing and sampling  the leverage scores is $\mc{O}(nr)$
with an additional $\mc{O}(r^2\cls)$ for factorizing $S^\top W$.
The hybrid algorithm also requires $\mc{O}(nr)$ for computing the leverage
scores and sampling. In the second stage, sRRQR is applied to a matrix of size
$r\times \cls$; this cost is $\mc{O}(r^2\cls)$, and is
independent of $n$. An additional cost  for factorizing $S^\top W$ is
$\mc{O}(r^3)$. This is summarized in \cref{tab:summary}. In summary, the hybrid
approach is both computationally efficient compared to other point selection
methods and has comparable DEIM error constants.

\begin{table}[!ht]\centering
\caption{Summary of various point selection techniques. Here $0 < \epsilon, \delta < 1 $ are user-defined parameters. We take the parameter $\beta = 1/2$. For the computational cost, terms that do not depend on $n$ are not considered. A note about the ${}^*$ entries: the corresponding DEIM error constants, each hold independently with probability at least $1-\delta$.}\label{tab:summary} 
\begin{tabular}{c|c|c|c|c}
 Method& \# indices  & Comp. cost  & DEIM error constant  & Reference \\ \hline
 DEIM  & $r$ & $\mc{O}(nr^2 )$ &  $\mc{O}(\sqrt{nr} 2^r)$ & \cite{sorensen2016deim}\\
 PQR & $r$ &$\mc{O}(nr^2)$ & $ \mc{O}(\sqrt{n}2^r)$ & \cite{drmac2016new}\\
 sRRQR  & $r$ & $\mc{O}(nr^2)$ &  $\mc{O}(\sqrt{nr})$ & \cite{drmac2017discrete} \\ 
 LS$^*$ & $\mc{O}\left(\frac{r\log (r/\delta)}{\epsilon^2}\right)$ & $\mc{O}(nr )$& $\mc{O}\left(\sqrt{\frac{n\epsilon^2}{r\log (r/\delta) (1-\epsilon)}}\right)$ & \cref{thm:sampling}  \\ 
 Hybrid$^*$  & $r$  & $\mc{O}(nr )$ & $\mc{O}\left(\sqrt{\frac{nr}{(1-\epsilon)}}\right)$  & \cref{thm:sampling} 
\end{tabular}
\end{table}

The important point is that the error constants of the hybrid point selection
approach are comparable with the best known deterministic bounds (using sRRQR);
however, the computational cost is far less than that of sRRQR, or the other
deterministic approaches. We advocate the hybrid approach since it has
reasonable computational cost, is accurate, and selects exactly $r$ indices
from the nonlinear function.

\subsection{Randomized point selection with randomized DEIM basis} Thus far, we
have described randomized point selection techniques assuming the availability of the standard DEIM
basis $W$. If the standard basis $W\in \R^{n\times r}$ is computationally
intensive to compute, we can alternatively use an R-DEIM basis (obtained for
example using \cref{alg:basic} or \cref{alg:subspace}). The leverage scores,
and the corresponding sampling probabilities $\{\pi^\beta_j\}_{j=1}^n$,  are
now computed corresponding to the basis $\Wh$ rather than the standard DEIM
basis $W$. To determine the selection operator, one may use either randomized
point selection technique---\cref{alg:lev} or \cref{alg:hyb}. In the following
result, we quantify the error in the resulting DEIM approximation---the main
challenge is now there are two sources of randomness: the sampling matrix
$\Omega$ as well as the sampling strategy that determines the selection
operator $S$.
\begin{theorem}
Let $\Wh\in \R^{n\times r}$ be obtained using \cref{alg:subspace} with oversampling parameter $p \geq 2$, and let $\Sh \in \R^{n\times r}$ obtained using the hybrid point selection algorithm \cref{alg:hyb} to define the randomized DEIM operator $\widehat\D_\text{Hy} = \Wh (\Sh^\top \Wh)^\dagger \Sh^\top$. Consider the same assumptions as in \cref{thm:expect}. Let {$0 < \delta < 1/2$} be a user-defined parameter. With probability at least $1-2\delta$
\[ \|f-\widehat\D_\text{Hy} f\|_2 \leq D_\text{Hy}\left( \|(I_n-P_W)f\|_2 + \frac{\gamma^{2q+1}}{1-\gamma}C_d\|P_Wf\|_2 \right) \]
where $D_\text{Hy}$ was defined in \cref{thm:sampling} and 
\[ C_d \equiv \frac{e\sqrt{r+p}}{p+1}\left(\frac{2}{\delta}\right)^{1/(p+1)}\left( \sqrt{n_s-r} + \sqrt{r+p} + \sqrt{2\log\frac{2}{\delta}}\right).\]
\end{theorem}
\begin{proof}
Define the event
\[ \mc{E} = \left\{ \Omega \left| \sin \theta_{\max} \leq \frac{\gamma^{2q+1}}{1-\gamma} C_d \right. \right\}.\]
Combining \cite[Theorems 4 and 6]{saibaba2018randomized}, the probability of the complementary event satisfies $\mb{P}(\mc{E}^c) \leq \delta$. Similarly, define the event \[ \mc{F}  = \left\{ \Sh, \Omega \left| \|f- \D_\text{Hy}f\|_2 > D_\text{Hy} \left(\| (I_n-P_W)f\|_2 + \frac{\gamma^{2q+1}}{1-\gamma} C_d \|P_Wf\|_2\right)\right. \right\}.\]
By \cref{thm:sampling}, $\mb{P}(\| \D_\text{Hy}\|_2 \leq D_\text{Hy} |\mc{E}) \geq 1-\delta$ \changes{and $\rank(\D_\text{Hy}) = r$. The assumption $1 \leq r < \rank(A)$ ensures $\D_\text{Hy} \neq 0$ and $\D_\text{Hy} \neq I_n.$ } By \cref{thm:error1}, we have
\[ \mb{P}\left( \|f-\D_\text{Hy} f\|_2 \leq D_\text{Hy}\left( \left. \|(I_n-P_W)f\|_2 + \frac{\gamma^{2q+1}}{1-\gamma}C_d\|P_Wf\|_2 \right)\right| \mc{E}  \right) \geq 1-\delta,\]
or the complementary event $\mc{F}$ satisfies $\mb{P}(\mc{F}|\mc{E}) \leq \delta$. By the law of total probability
\[ \begin{aligned}
\mb{P}(\mc{F}) = & \mb{P}(\mc{F}|\mc{E})\mb{P}(\mc{E}) + \mb{P}(\mc{F}|\mc{E}^c)\mb{P}(\mc{E}^c)  
\leq     \mb{P}(\mc{F}|\mc{E}) + \mb{P}(\mc{E}^c).
\end{aligned}
\]
Therefore, $\mb{P}(\mc{F}) \leq \delta + \delta = 2\delta$. The complementary event satisfies the advertised bound. 
\end{proof}
A similar result can be derived for the leverage score approach but we omit the details.

\section{Numerical Experiments}\label{sec:num} In \cref{ssec:ex1} we
investigate the accuracy of the DEIM basis generated using the randomized
algorithms discussed in \cref{sec:randbasis}. In \cref{ssec:ex2} we investigate
the performance of the randomized point selection algorithms proposed in
\cref{sec:point}. In \cref{ssec:ex3} we apply these randomized algorithms to a
large-scale PDE-based application. All the timing results were computed on a
computing cluster in which each node has an Intel(R) Xeon(R) CPU E5-2690 processor, with 8-core
CPUs at 2.90GHz and 128GB of DDR3 RAM. The code was implemented and tested in \textsc{MATLAB} 2018a and the operating system was Ubuntu 16.04.

\subsection{Example 1: Randomized range finder}\label{ssec:ex1} In our first example, we
consider the setup of the synthetic example in~\cite[section
2.3]{peherstorfer2014localized}. The spatial domain and the parameter domain are both taken to be $\mc{D}_s =
[0,1]^2$. We define the  function $g $ as
\[ g(x_1,x_2;\mu_1,\mu_2) \equiv \frac{1}{\sqrt{h(x_1;\mu_1) + h(x_2;\mu_2) +
		0.1^2 }}.\]  
where $h(z;\mu) = ((1-z)-(0.99\cdot\mu-1))^2 $. The function that
is to be interpolated is 
\begin{equation}\label{eqn:pde} \begin{aligned}
 f(x_1,x_2;{\mu_1,\mu_2}) = &  \>g(x_1,x_2;\mu_1,\mu_2) + g(1-x_1,1-x_2; 1-\mu_1, 1-\mu_2) \\
& + \quad g(1-x_1,x_2; 1-\mu_1,\mu_2) + g(x_1,1-x_2; \mu_1, 1-\mu_2). 
\end{aligned}
\end{equation}

Depending on the parameter $\mu$, the function $f$ has a sharp peak in one of the four
corners of $\mc{D}_s$. The function is discretized on a $100\times 100$ grid in
$\mc{D}_s$ with $n=10,000$, and parameter samples are drawn from a $25\times 25$ equispaced grid
in $\mathcal{D}$. These $n_s = 625$ snapshots are stored in the snapshot matrix $A$ and are used to construct the DEIM
approximation.

\begin{figure}[!ht]\centering
\includegraphics[scale=0.35]{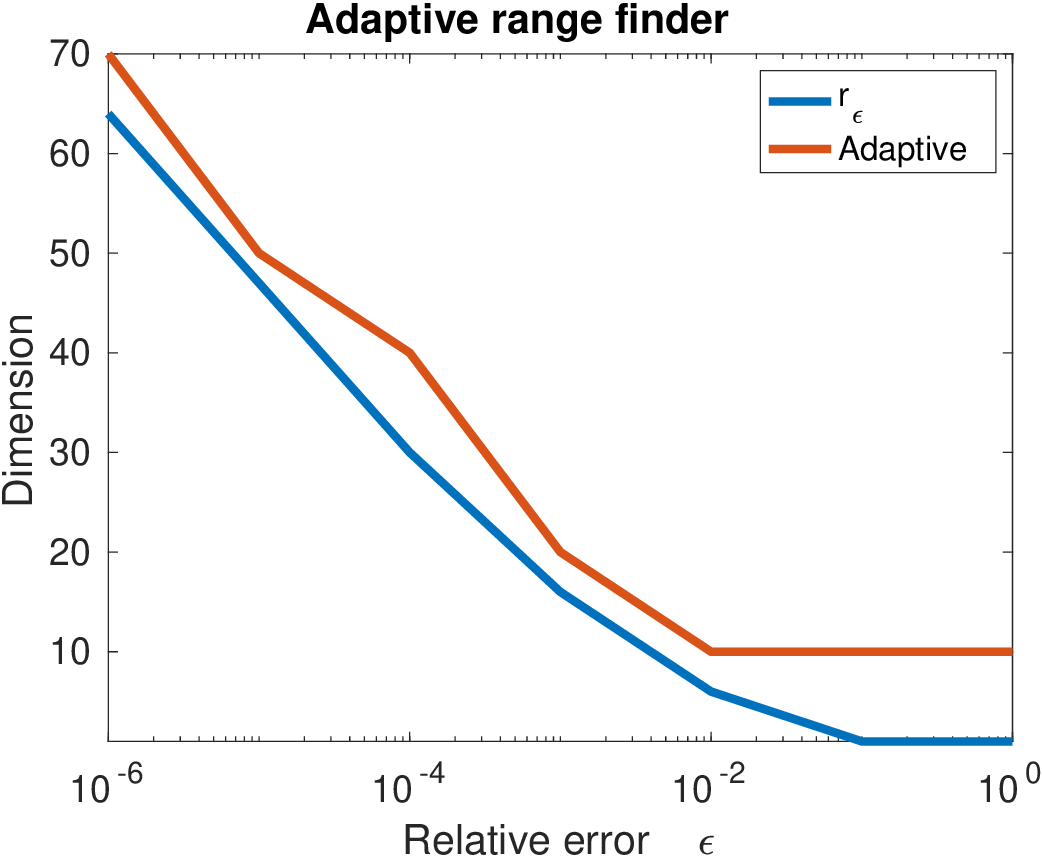}
	\caption{The dimension of the basis returned by the adaptive range
		finding algorithm compared with the truncation indices based on a cutoff tolerance. }
	\label{fig:adapt_rand}

\end{figure}

\subsubsection{Adaptive range finder} If the target subspace dimension is not
known \textit{a priori}, we use the adaptive procedure outlined in
\cref{ssec:adapt}. With the same settings as the previous example,  we use
\cref{alg:rangetwo} for determining the range. For comparison, we consider the
standard DEIM basis. Given a user-defined tolerance $\epsilon > 0$, the
dimension of the standard DEIM basis $W$ (assuming $n\neq n_s$) is taken to be 
\[
r_\epsilon(A) =  \min \left\{ r  \left| \sum_{k=r+1}^{n_s} \sigma_{k}^2(A) > \epsilon \sum_{k=1}^{n_s} \sigma_{k}^2(A) \right. \right\},
\]
since this ensures that $W$ satisfies 
\[ \| (I_n-WW^\top) A\|_F^2 \leq \epsilon \| A \|_F^2 .\]
Note that we say the dimension of a basis, even though dimension is an attribute of the subspace spanned by the basis vectors.

The standard DEIM basis has the smallest dimension satisfying the above
equality; this follows from the optimality of the SVD. Therefore, the dimension
of the basis returned by \cref{alg:rangetwo} must be at least as large as the
dimension of the standard DEIM basis. In this experiment, we investigate how
close these two dimensions are.  The tolerance $\epsilon$ is varied as 
 $\epsilon \in \{10^{2},\dots,10^{-6}\}$. For the adaptive algorithm, the
block size was set to be $10$, and the maximum iterations was taken to be $40$.
In \cref{fig:adapt_rand}, we compare the two dimensions depending on the cutoff
tolerance $\epsilon$. From the figure, it is clear that as the tolerance
$\epsilon$ decreases, both the dimensions increase.  Second,  it can be seen
that the both the dimensions are in good agreement, and the dimension returned
by the adaptive randomized algorithm is only slightly larger than $r_\epsilon$,
demonstrating that it can be used in real applications. The accuracy of the
R-DEIM basis is further investigated in \cref{sssec:rdeim}. 

\subsubsection{Subspace iterations $q$ and oversampling parameter $p$} In our
next experiment, we investigate the effects of the number of subspace
iterations $q$ and the oversampling parameter $p$. \cref{thm:expect} guarantees
that with increasing iterations $q$ the factor $\gamma^{2q+1}$ subdues the
influence of the constant $C$, since $\gamma$ is assumed to be less than $1$.
Similarly, note that the oversampling parameter $p$ appears in the denominator
of $C$; therefore, increasing the oversampling parameter $p$ results in more
accurate subspace computation. Both~\cite{gu2015subspace,halko2011finding}
recommend choosing $p=10-20$. We found this choice of parameters to be
satisfactory in our experiments and  \cref{fig:ex1_subspaces} confirm these
findings.

\begin{figure}[!ht]\centering
	\includegraphics[scale=0.3]{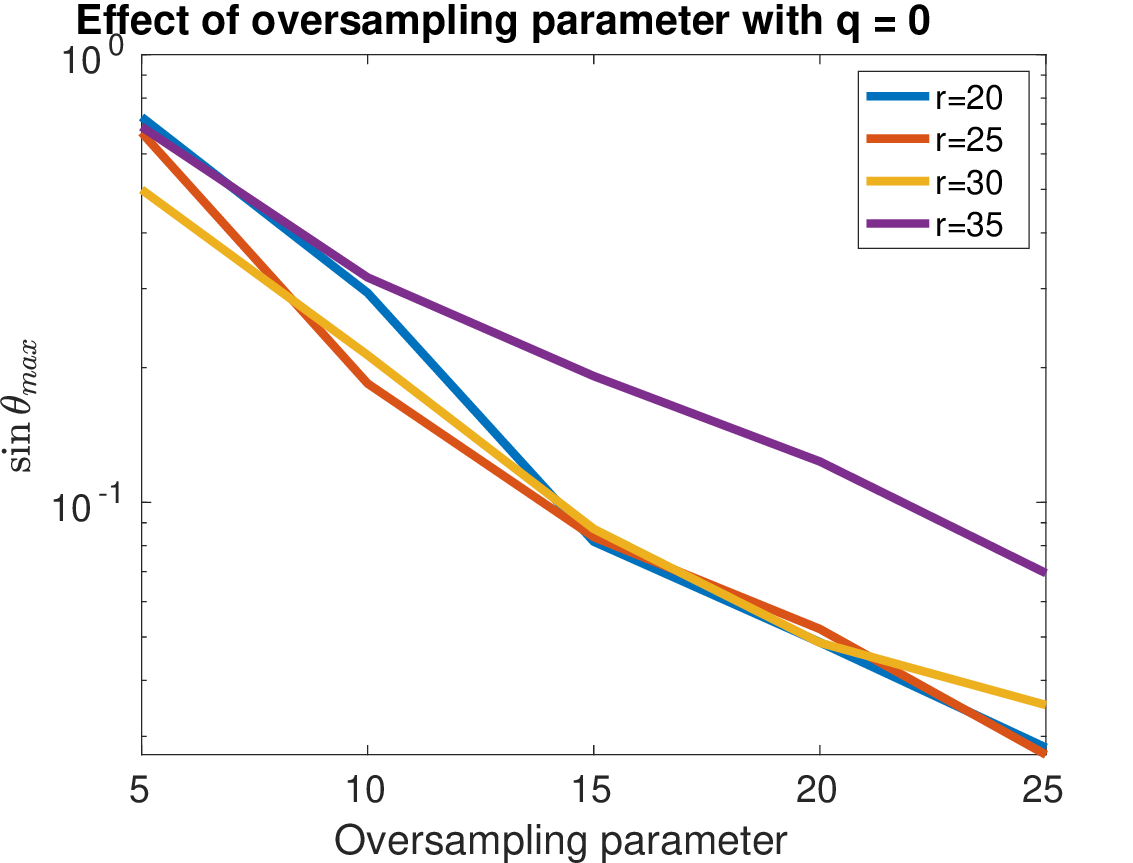}
\includegraphics[scale=0.3]{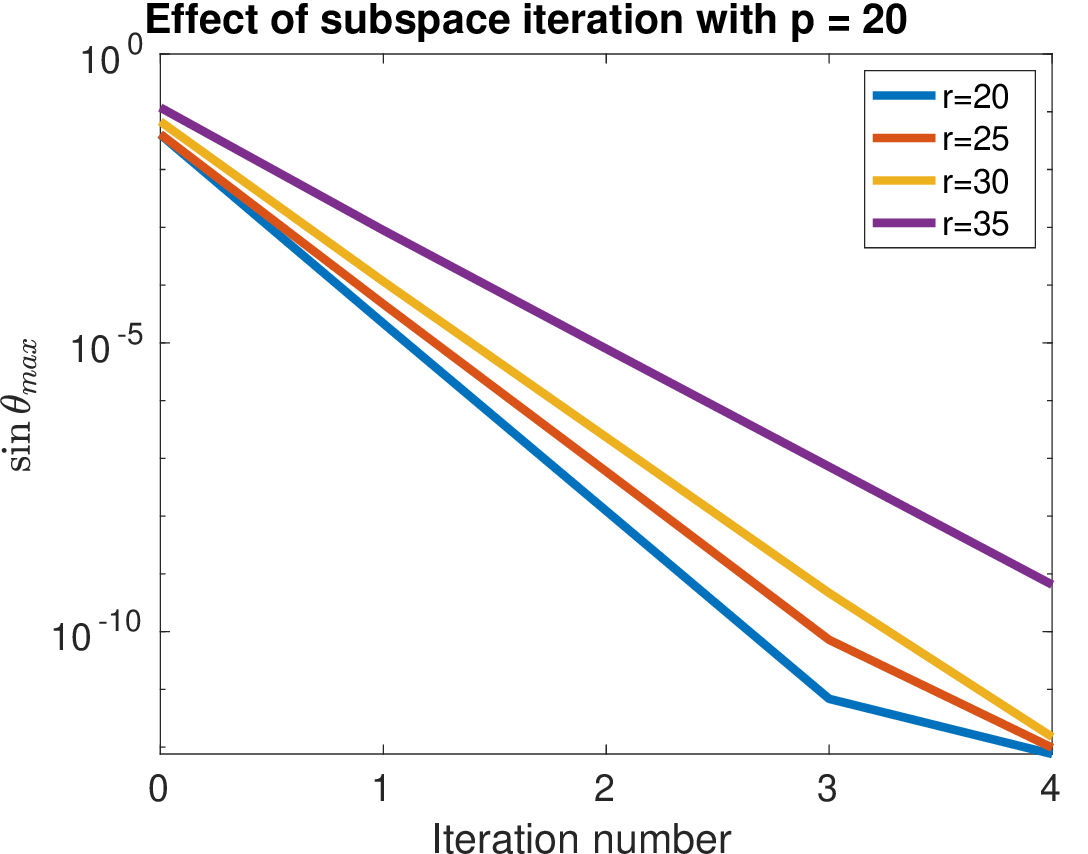}
	\caption{Example 1 -  {Effect of  (left) increasing oversampling parameter $p \in \{5,10,15,20,25\}$ (we fix $q=0$), (right)  increasing subspace iterations $q \in \{0,1,2,3,4\}$ (we fix $p=20$), on the accuracy of the R-DEIM basis measured as $\sin\theta_\text{max} = \|P_W - P_{\Wh}\|_2$. } }
	\label{fig:ex1_subspaces}
\end{figure}

\begin{figure}[!ht]\centering
	\includegraphics[scale=0.35]{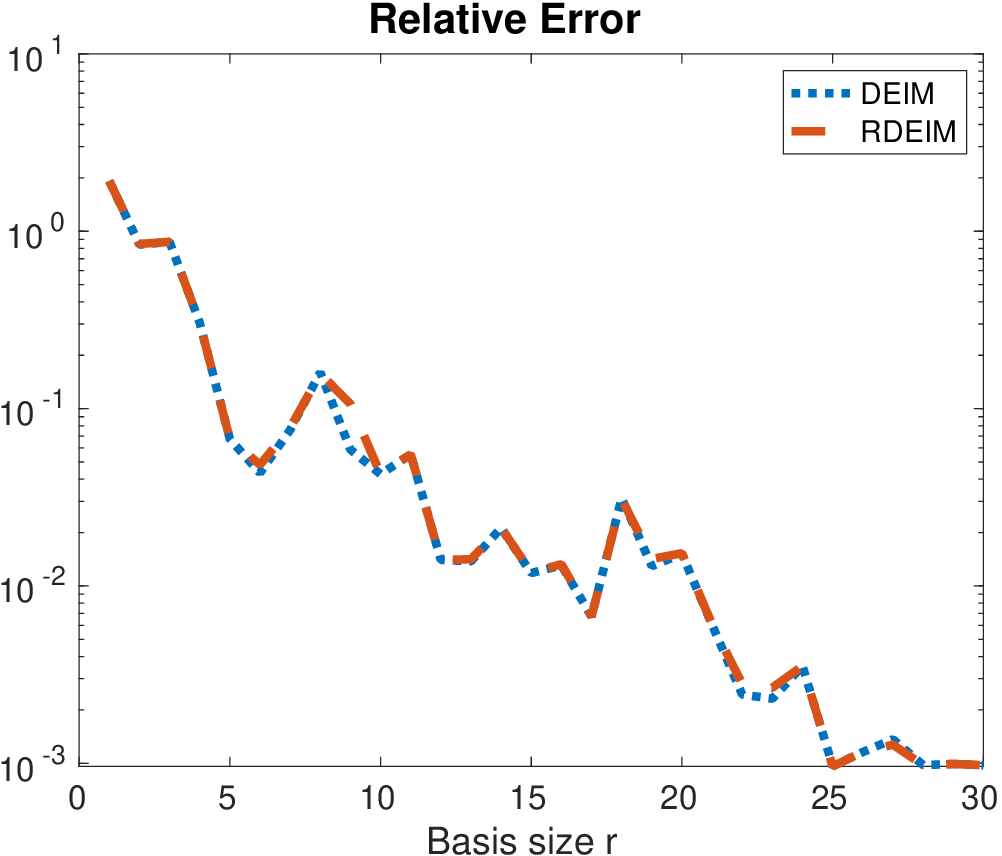}
\includegraphics[scale=0.35]{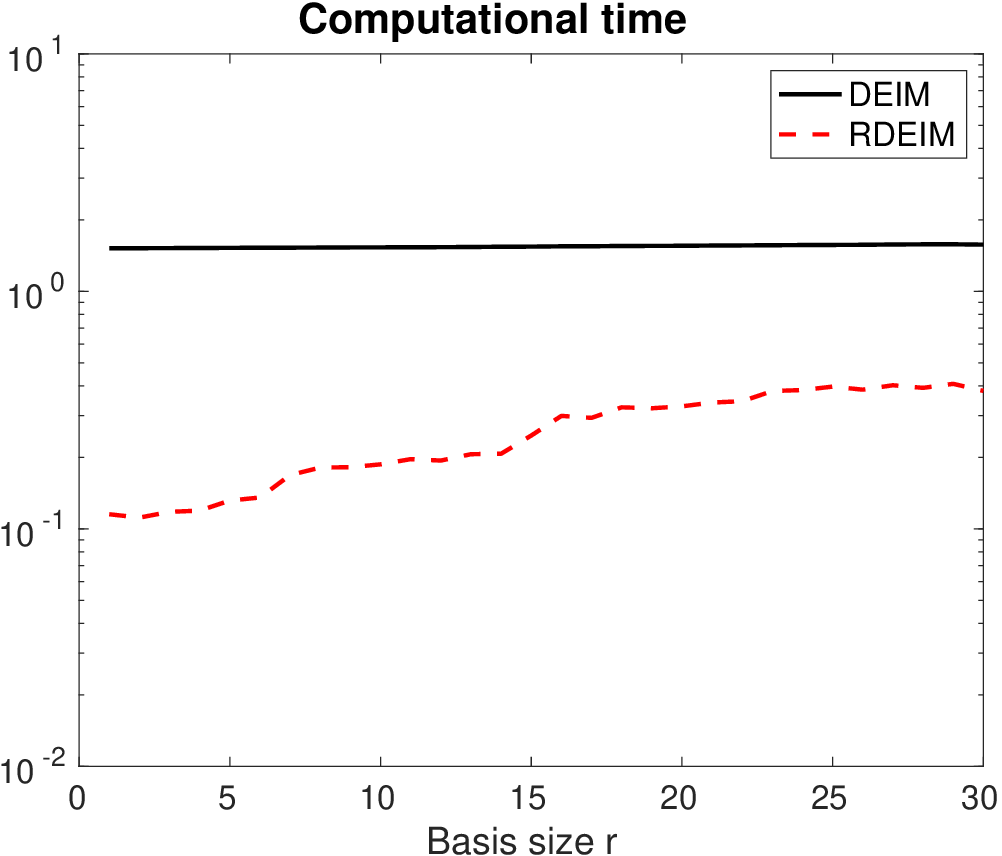}
	\caption{Example 1 - (left) Relative error using the DEIM and the
	R-DEIM approximation plotted against basis dimension $r$. R-DEIM basis was computed using \cref{alg:basic} with $r=30$ and oversampling parameter $p=10$. {(right) Timing results using DEIM and R-DEIM plotted against basis dimension $r$.} }
	\label{fig:ex1_error}
\end{figure}

\subsubsection{Accuracy of the R-DEIM basis}\label{sssec:rdeim} In this experiment, we
compare the accuracy of the R-DEIM basis with the standard DEIM basis for
Example 1. For the R-DEIM approximation, we use an oversampling parameter
$p=10$. The dimension of the R-DEIM basis is varied until $r=30$, and the error
is compared with the standard DEIM approximation. We used the PQR algorithm for
point selection. To account for the randomness, the resulting error in the
R-DEIM approximation  was averaged over $100$ runs.  As can be seen from
\cref{fig:ex1_error}, the error in the R-DEIM approximation is comparable to the
error in the standard DEIM approximation. We also see the time for computing the R-DEIM approximation is far lower than the time for the DEIM approximation.


\subsection{Example 2: Point selection}\label{ssec:ex2} This example is a continuation of the
Example 1, but we now focus on the randomized point selection. For this
example, we use the standard DEIM basis, i.e., the basis computed using the
left singular vectors of the snapshot matrix $A$. 

We first compare the two randomized point selection techniques---leverage scores
(LS), and the hybrid point selection algorithms---with the  PQR
algorithm.  Recall that theory suggests that the we have to choose the number
of samples according to the formula in \cref{thm:sampling}. Suppose we choose
the parameters $\epsilon = 0.99$, which ensures  $1/\sqrt{1-\epsilon} = 10$ and
$\delta = 0.01$. Our numerical experiments showed that the number of samples
required by the LS point selection algorithm appear to be an overestimate. In fact, the
number of samples can sometimes exceed $n$---which is antithetical to the spirit of the DEIM
approximation.  In practice, we found $\lceil 3r\log r \rceil$ samples
sufficient to provide accurate approximations. We used the same number of
samples for the hybrid point selection algorithm.

\begin{figure}[!ht]\centering
	\includegraphics[scale=0.35]{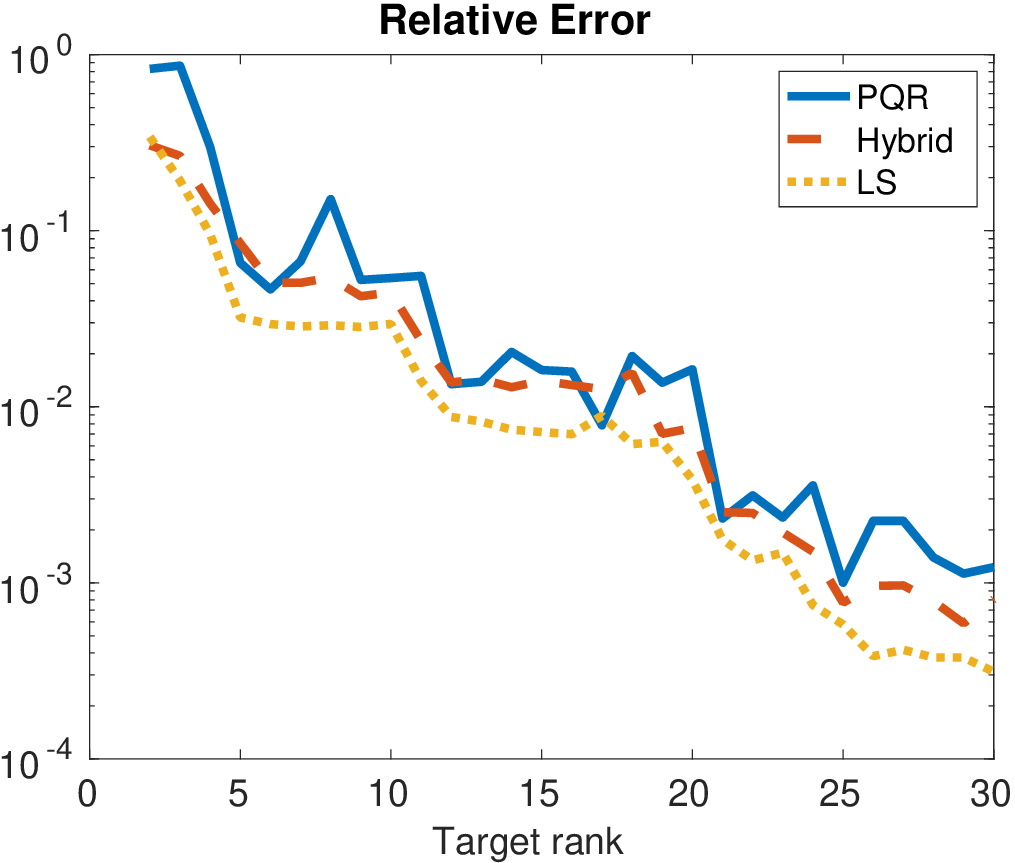}
	\includegraphics[scale=0.35]{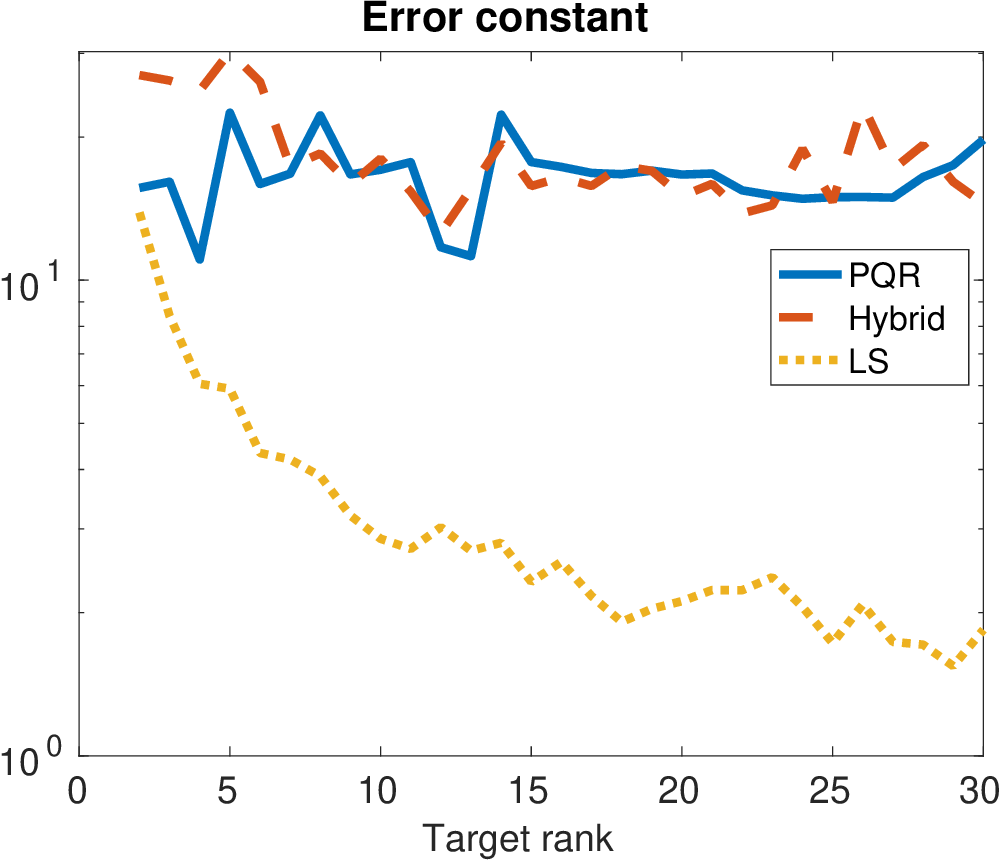}
\includegraphics[scale=0.35]{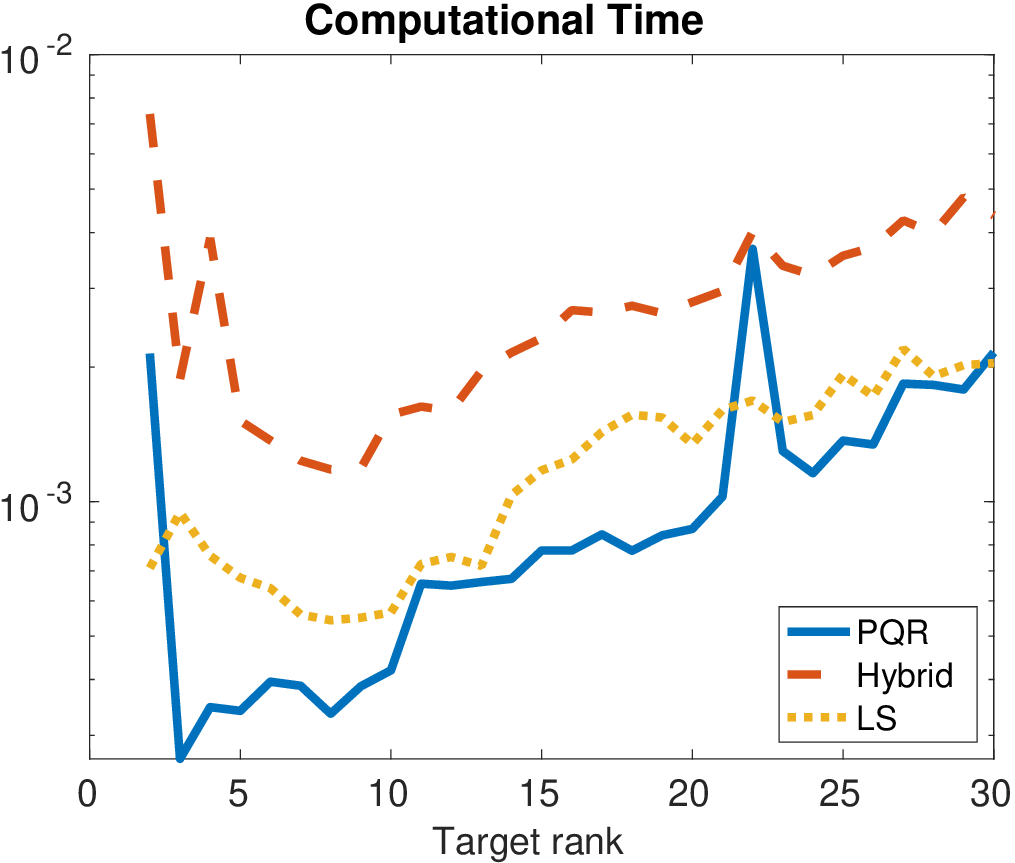}
	\caption{Example 2 - (left) Error in the DEIM approximation,  
	(right) the DEIM error constants $\|\D\|_2 $ using 
		the PQR algorithm and the hybrid approach \cref{alg:hyb}, 
		(bottom) the computational time for all the point selection methods.}
	\label{fig:example2_err}
\end{figure}

 In \cref{fig:example2_err}, we compare the error in the DEIM approximations
and the corresponding error constants $\|\D\|_2$ for the various algorithms.
The error constant is much smaller for the LS point selection algorithm
compared to both the hybrid and deterministic approaches. This is because the
number of samples in the LS point selection algorithm is much larger than $r$.
However, the accuracy of the DEIM approximation is comparable for all three
methods, and does not appear to be significantly affected by the choice of the
point selection method. {The computational time for each point selection method
is also reported, the timings were averaged over $10$ runs.}
 
\begin{figure}[!ht]\centering
	\includegraphics[scale=0.35]{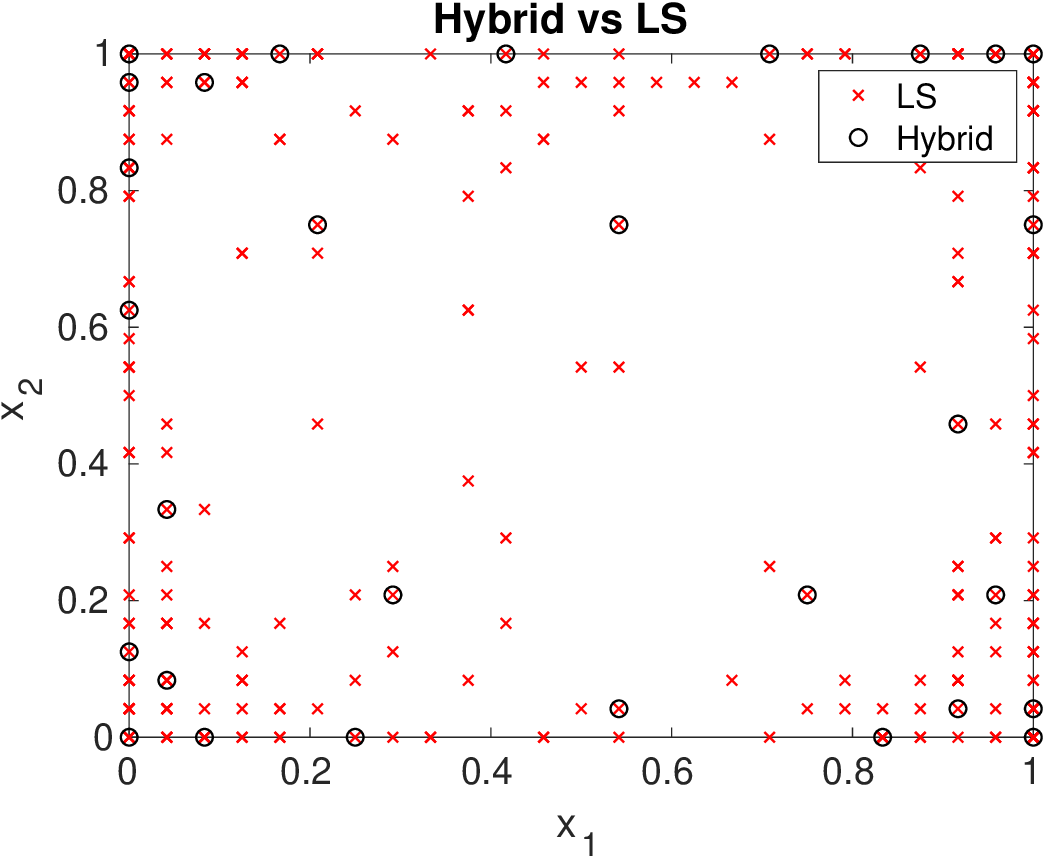}
	\includegraphics[scale=0.35]{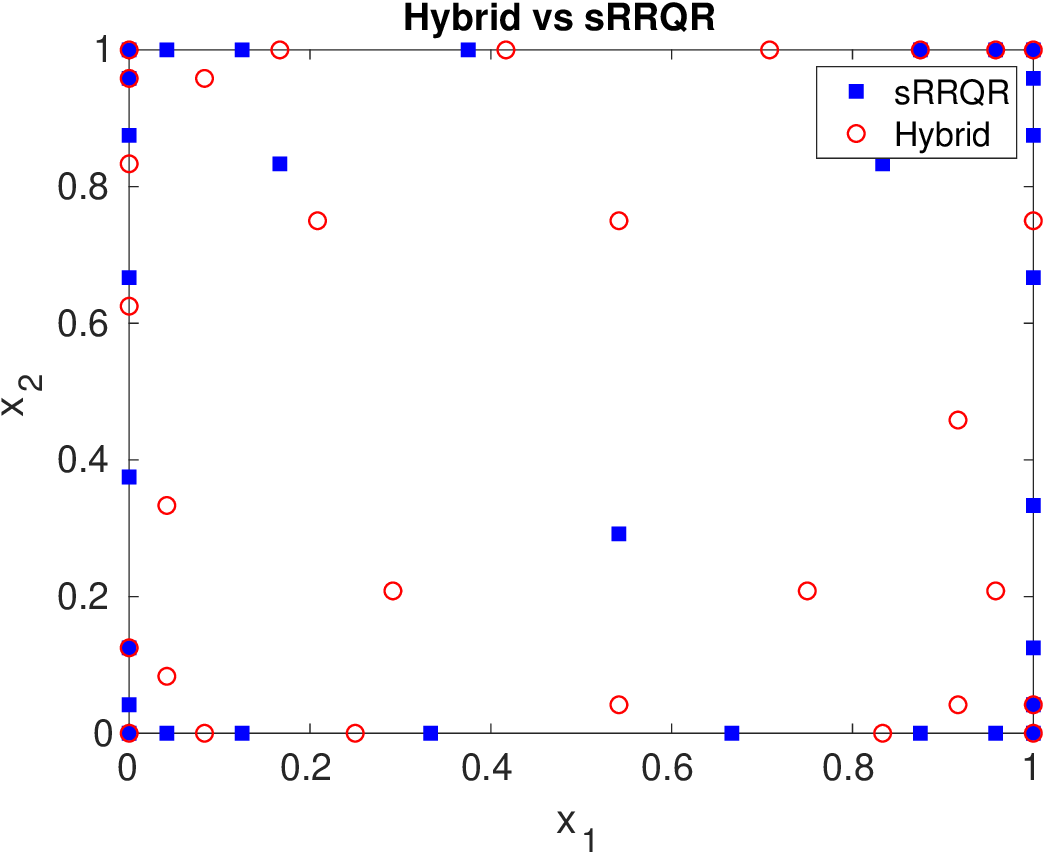}
	\caption{Example 2 - The indices chosen by the various point selection methods.}
	\label{fig:example2_comp}
\end{figure}

To provide more insight into the hybrid point selection algorithm, we compare
it with sRRQR and LS point selection algorithms. In \cref{fig:example2_comp}, 
we plot the point selected by the LS sampling approach on the left panel. By construction, the
hybrid approach subsamples from the points selected by the LS algorithm, and
these points are overlaid in the figure. In the right panel of the same figure,
the hybrid points are compared with those selected by sRRQR. It it interesting
to note that several points overlap between the hybrid and the sRRQR
approaches, even though they were selected using different algorithms. 

Our conclusion is that the proposed hybrid point selection algorithm is a good
compromise between the deterministic (e.g., PQR, sRRQR) and randomized point
selection algorithms.  Compared to the deterministic algorithms it is
computationally efficient and has comparable accuracy. The hybrid point selection algorithm
also has comparable accuracy to the LS algorithm but is advantageous since only
$r$ components of the nonlinear function need to be evaluated. 

\subsection{Example 3: PDE-based application}\label{ssec:ex3} We consider a example
from~\cite[section 8.4]{quarteroni2015reduced} involving a parameterized
advection-diffusion PDE that models, for example, the evolution of the
concentration of a pollutant. Consider the following PDE defined on a domain
$\mathcal{D}=[0,1]^2$ with boundary $\partial \mathcal{D}$
\begin{eqnarray}
-\mu_1 \Delta u  + \boldsymbol{b}(\mu_2) \cdot \nabla u + a_0 u = &  s(\mathbf{x};\boldsymbol{\mu}) \quad & \mathbf{x} \in \mathcal{D}\\
\mu_1 \mathbf{n}\cdot \nabla u = & 0 & \mathbf{x} \in \partial \mathcal{D}.
\end{eqnarray}
Here, $\boldsymbol{\mu} = \bmat{\mu_1 & \dots & \mu_5}$, $a_0$ is a positive
constant, and $\mathbf{n}$ is the normal vector. The wind velocity is taken to
be  $\boldsymbol{b}(\mu_2) = [\cos\mu_2,\sin\mu_2]$, which is a constant in
space but depends nonlinearly on the parameter $\mu_2$. The source term
$s(\boldsymbol{\mu})$ has the form of a Gaussian function centered at
$(\mu_3,\mu_4)$ and spread $\mu_5$, i.e.,
\[ s(\mathbf{x};\boldsymbol{\mu}) = \exp\left( -\frac{(x_1-\mu_3)^2 + (x_2-\mu_4)^2}{\mu_5^2}\right).\]
The cost of solving the PDE for many different values of $\boldsymbol{\mu} $ is
high and therefore, it is computationally beneficial to develop a reduced order
model that makes the online solution of the PDE feasible. The nonlinear
dependence on the parameters arises from the source term $s(\mathbf{x}; \boldsymbol\mu)$.
To tackle this, we use the POD-DEIM approach.

\begin{figure}[!ht]\centering
	\includegraphics[scale=0.35]{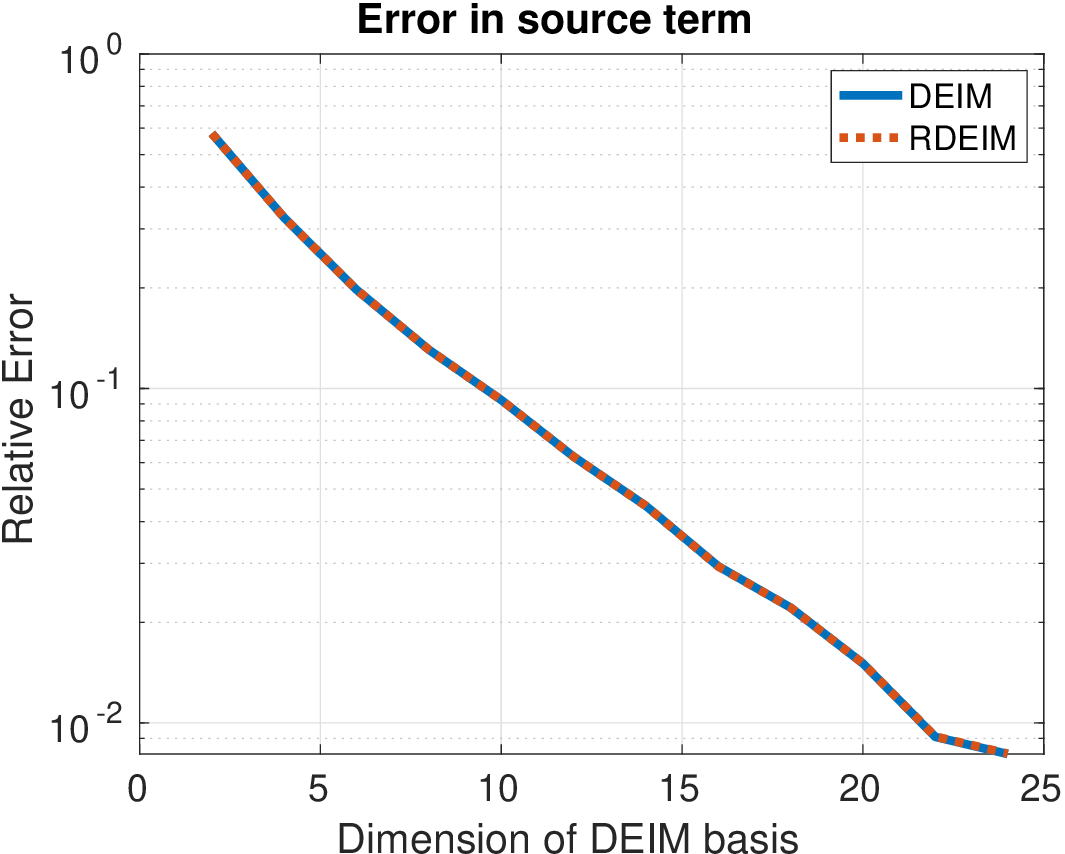}
	\includegraphics[scale=0.35]{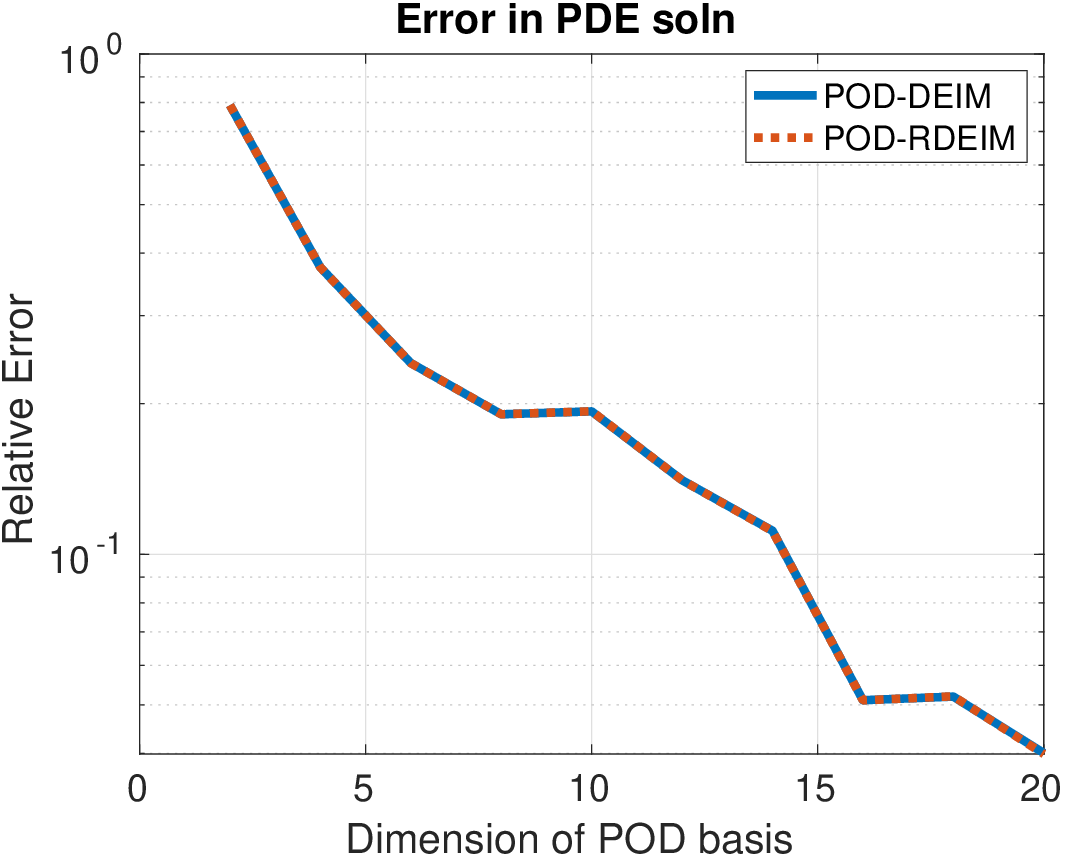}
	\caption{ (left) The error in the DEIM and R-DEIM approximation for the source term $s(\mathbf{x};\boldsymbol{\mu})$. The R-DEIM basis was constructed using \cref{alg:basic}, the oversampling parameter was $p=20$. We used the Hybrid algorithm for the point selection. (right) The error in the DEIM and R-DEIM approximation {for the PDE}.}
	\label{fig:ex3deim}
\end{figure}

Our first experiment is similar to that in~\cite[section
10.5.1]{quarteroni2015reduced}. We first consider the cost of approximating the
source term $s(\mathbf{x};\boldsymbol{\mu})$ over the range of parameters
$\mu_3 \in [0.2,0.8]$, $\mu_4 \in [0.15,0.35] $ and $\mu_5$  is chosen between
$ [0.1,0.35]$. A training set for  $\boldsymbol\mu$ is generated by Latin
hypercube sampling with $n_s = 1000$ training points. Two different
approximations are generated using DEIM and R-DEIM. The number of DEIM basis
vectors used were $r = 24$ and was determined based on the singular value decay
of the snapshot matrix.  The R-DEIM basis is also fixed to be of size $r = 24$
and was computed using \cref{alg:basic} with an oversampling parameter $p=20$.
The error is computed by averaging over $200$ different randomly generated test
points. The results are displayed in \cref{fig:ex3deim}. The hybrid point
selection algorithm is used for both the standard DEIM basis and the randomized
DEIM basis---the same parameters are used as in \cref{ssec:ex2}. As can be seen
the error between the two different methods are comparable.  For this
application, the number of quadrature nodes is {$165,888$. It is worth
mentioning that the CPU time for computing the compact SVD is $7.67$ seconds
whereas the CPU time for the  randomized range finder is $0.25$ seconds.
The computational time for the point selection was comparatively
$0.06$ seconds. The speedup is more impressive for larger problems;
see~\cite{bach2019fixed} for more detailed comparisons on large-scale
problems.}


Our second experiment is similar to~\cite[section
10.5.2]{quarteroni2015reduced}, in which the parameters $\mu_1 = 0.03$ and
{$\mu_5=0.25$} are fixed and the remaining parameters are taken to be in the range
$\mu_2 \in [0,2\pi]$, $\mu_3\in[0.2,0.8]$ and $\mu_4\in[0.15,0.35]$. The goal
is then to compute a ROM for solving the PDE~\cref{eqn:pde} for the above
chosen range of parameters. The source term is approximated using the DEIM and
the R-DEIM approaches, as described in the previous experiment (the dimension
$r=24$ was used for both approximations. The POD algorithm is computed using
$1000$ snapshots and the POD dimension $k=20$ is used to construct the reduced
basis space. The error is shown in the right panel of \cref{fig:ex3deim}, where the POD-R-DEIM
approximation is compared against the POD-DEIM approximation. It is seen that
both the errors are comparable which validates our approach.

\section{Conclusions} We have provided randomized algorithms for tackling the
two main bottlenecks of the standard DEIM algorithm. First, we propose several
randomized algorithms for approximately computing the DEIM basis and highlight
various benefits of the randomized algorithms including adaptivity.  Second, we
proposed two randomized point selection algorithms--- one based on leverage
score sampling, the other combines leverage score and rank-revealing
factorizations. We provide detailed analysis of the error in the resulting
algorithms that clearly show the trade-off between computational cost and
accuracy. The proposed algorithms are more efficient than the standard
techniques and have comparable accuracy.  Numerical experiments in
\cref{sec:num} confirm these findings and give insight into the choice of parameters.

\section{Acknowledgments} The author would like to thank  Ilse C.F.\ Ipsen and
Zlatko Drma{\v{c}} for discussions and Ivy Huang for her help with this manuscript.

\appendix
\bibliography{refs}
\bibliographystyle{abbrv}

\end{document}